\documentclass{amsart}
\usepackage{amsmath,amsfonts,amsthm}
\usepackage{amssymb,latexsym}
\usepackage{mathtools}
\usepackage[colorlinks]{hyperref}
\usepackage{graphicx}
\usepackage[all]{xy}
\usepackage{layout}
\usepackage{bbm}

\theoremstyle{plain}
\newtheorem{theorem}{Theorem}[section]
\newtheorem{lemma}[theorem]{Lemma}
\newtheorem{proposition}[theorem]{Proposition}
\newtheorem{corollary}[theorem]{Corollary}

\theoremstyle{definition}

\newtheorem{definition}[theorem]{Definition}
\newtheorem{example}[theorem]{Example}

\newtheorem{remark}[theorem]{Remark}

 \newcommand{\M}{\mathbb M_{k(n)}(E)}
 
 \newcommand\ssubset{
 	\mathrel{
 		\mathrlap{\subset}
 		\hphantom{\ll}
 		\mathllap{\subset}
 	}
 }
\begin{document}

\title{finite  approximation properties of
$C^{*}$-modules III}

\author{Massoud Amini}

\address{Department of Mathematics\\ Faculty of Mathematical Sciences\\ Tarbiat Modares University\\ Tehran 14115-134, Iran}
\email{mamini@modares.ac.ir, mamini@ipm.ir}







\maketitle

\begin{abstract}
We introduce and study a notion of module nuclear dimension  for a $C^{*}$-algebra $A$ which is $C^*$-module over another $C^*$-algebra $\mathfrak A$ 
with compatible actions. We show that the module nuclear dimension of $A$ is zero if $A$ is $\mathfrak A$-NF. The converse is shown to hold when $\mathfrak A$ is a $C(X)$-algebra with simple fibers, with $X$ compact and totally disconnected. We also introduce a notion of module decomposition rank, and show that when $\mathfrak A$ is unital and simple, if the module decomposition rank of $A$ is  finite then $A$ is $\mathfrak A$-QD. We study the set $\mathcal T_\mathfrak A(A)$ of $\mathfrak A$-valued module traces on $A$ and relate  the Cuntz semigroup of $A$ with lower semicontinuous affine functions on the set $\mathcal T_\mathfrak A(A)$. Along the way, we also prove a module Choi-Effros lifting theorem. We give estimates of the module nuclear dimension for a class of examples.   
\end{abstract}

\section{introduction} \label{1}

In this paper we continue the study of module version of finite  approximation properties of $C^*$-algebras as studied in \cite{bo}. Some of these,  including important notions such as nuclearity, exactness and weak expectation property (WEP) are extended to the context of $C^*$-algebras with compatible module structure in  \cite{a2}. In \cite{a3} we studied module versions of quasidiagonality and local reflexivity and introduce and studied a notion of amenability for vector valued traces.  Here we continue this study by defining a notion of module nuclear dimension. We define the class of  approximately finite modules as the objects with zero module nuclear dimension. 

A ``finite approximation'' scheme for $C^*$-modules is an approximately commuting diagram as follows,

\begin{center}
$\xymatrix @!0 @C=4pc @R=3pc { A \ar[rr]^{\theta} \ar[rd]^{\varphi_n} && B  \\ & \mathbb M_{k_n}(\mathfrak A) \ar[ur]^{\psi_n}}$
\end{center}

\noindent where $A$ and $B$ are $C^*$-algebras and right $\mathfrak A$-modules with compatible actions and $\theta$, $\varphi_n$ and $\psi_n$ are contractive completely positive (c.c.p.) module maps. The notion of module nuclear dimension for the map $\theta$ is defined as in the classical case, by requiring that given $\varepsilon>0$ and a finite subset $F\ssubset A$, in the above approximately commuting set of triangles, the module maps $\psi_n$ could be chosen to be decomposable into a sum of finitely many order-zero module maps.

The paper is organized as follows: In section \ref{2} we define the notion of module nuclear dimension and compare and contrast it with the existing notions of nuclear dimension and relative nuclear dimension. In section \ref{3} we introduce $\mathfrak A$-AF algebras and show that these are exactly the objects with zero module nuclear dimension. In the last section, we give estimates of the module nuclear dimension for a class of examples.

For the rest of this paper, we fix a $C^*$-algebra $\mathfrak A$ and let $A$ be a $C^*$-algebra and a right Banach
$\mathfrak A$-module (that is, a module with contractive right action) with compatible conditions,
\begin{equation*}
(ab)\cdot\alpha=a(b\cdot\alpha),\,\, a\cdot\alpha\beta=(a\cdot\alpha)\cdot\beta,
\end{equation*}
for each $a,b\in A $ and $\alpha, \beta\in \mathfrak A.$ In this case, we  say that $A$ is a (right) $\mathfrak A$-$C^*$-module, or simply a $C^*$-module (it is then understood that the algebra and module structures on $A$ are compatible in the above sense). A $C^*$-subalgebra which is also an $\mathfrak A$-submodule is simply called a $C^*$-submodule. We say that $A$ is a $*$-module if we have,
$$(a\cdot\alpha)^*=a^*\cdot\alpha^*,\ \ (a\in A, \alpha\in \mathfrak A).$$
If moreover, we have the compatibility condition,
$$(a^*\cdot\alpha^*)^*\cdot\beta=((a\cdot\beta)^*\cdot\alpha^*)^*,$$
for each $a\in A $ and $\alpha, \beta\in \mathfrak A,$ then if we define a left action by
$$\alpha\cdot a:=(a^*\cdot\alpha^*)^*,$$
then $A$ becomes an $\mathfrak A$-bimodule with compatibility conditions,
\begin{equation*}
\alpha\cdot (ab)=(\alpha\cdot a)b,\,\, \alpha\beta\cdot a=\alpha\cdot(\beta\cdot a),\, \, \alpha\cdot(a\cdot \beta)=(\alpha\cdot a)\cdot \beta,
\end{equation*}
for each $a,b\in A $ and $\alpha, \beta\in \mathfrak A.$ In this case, there is a canonical $*$-homomorphism from $\mathfrak A$ to the multiplier algebra $M(A)$ of $A$, sending $\alpha$ to the pair $(L_\alpha, R_\alpha)$ of left and right module multiplication map by $\alpha$. If the action is {\it non degenerate}, in the sense that, given $\alpha$, $a\cdot \alpha=0$, for each $a\in A$, implies that $\alpha=0$ (and so the same for the left action), then the above map is injective and so an isometry, and we could (and would) identify $\mathfrak A$ with a $C^*$-subalgebra of $M(A)$. If this holds for a $*$-module, this forces $\mathfrak A$ to be commutative, say $\mathfrak A\simeq C_0(X)$, and it could be mapped onto a $C^*$-subalgebra of $Z(M(A))$, that is, $A$ is a $C_0(X)$-algebra \cite[page 354]{wi}. 

We say that a two sided action of $\mathfrak A$ on $A$ is a {\it biaction} if the right and left actions are compatible, i.e.,
$$(a\cdot\alpha)b=a(\alpha\cdot b)\ \ (\alpha\in\mathfrak A, a,b\in A).$$
When the action is non degenerate and $\mathfrak A$ acts on $A$ as a $C^*$-subalgebra of $M(A)$, we have a biaction.

In some cases we have to work with operator $\mathfrak A$-modules with no algebra structure (and in particular with certain Hilbert $\mathfrak A$-modules). In such cases, we only assume the second compatibility condition involving multiplication of elements in $\mathfrak A$. If $E, F$ are operator $\mathfrak A$-modules, a module map $\phi: E\to F$ is a continuous linear map which preserves the right $\mathfrak A$-module action.

Throughout this paper, we use the notation $\mathbb B(X)$ to denote the set of bounded adjointable linear operators on an Hilbert $C^*$-module  $X$.

\section{module nuclear dimension} \label{2}

In this section, $\mathfrak A$ is a $C^*$-algebra and $A$, $B$  are right $\mathfrak A$-$C^*$-modules satisfying the compatibility conditions of previous section. When the action is non degenerate, we identify $\mathfrak A$ with a $C^*$-subalgebra of the multiplier algebra $M(A)$ (or that of $A$, when $A$ is unital).

The representations of $\mathfrak A$-$C^*$-modules are defined on $\mathfrak A$-correspondences. An $\mathfrak A$-{\it correspondence} is a right Hilbert $\mathfrak A$-module $X$ with a left action of $\mathfrak A$ via a representation of $\mathfrak A$ into $\mathbb B(X)$. A {\it representation} of a $\mathfrak A$-bimodule $A$ in $X$ is a $*$-homomorphism from $A$ into $\mathbb B(X)$, which is also a right $\mathfrak A$-module map with respect to the canonical right $\mathfrak A$-module structure of $\mathbb B(X)$ coming from the left  $\mathfrak A$-module structure of $X$.

The notion of nuclear dimension of  C*-algebra is defined by Winter and Zacharias \cite{wz} as a non-commutative generalisation of topological covering dimension, by modifying the earlier notion of decomposition rank \cite{kw}. This later notion, as for the older notions of stable rank \cite{r} and real rank \cite{bp} proved to be useful, but somewhat restrictive. For instance, the notion of decomposition rank is used in the classification of stably finite, separable, simple, nuclear C*-algebras, but it restrictive, as a C*-algebra of finite decomposition rank must be quasidiagonal \cite[Theorem
4.4]{kw}. 

 We barrow the notion of admissible maps from \cite[Definition 2.1]{a2}.

\begin{definition} \label{adm}
	The class of admissible c.p. maps between operator $\mathfrak A$-modules is characterized through the following set of rules:
	
	$(i)$ a c.p. module map is admissible,
	
	$(ii)$ if a c.p. map $\phi: A\to B$ is admissible, then so are the maps $x\mapsto \phi(uxu^*)$ and $x\mapsto v\phi(x)v^*$, for each $u\in A$ and $v\in B$,
	
	$(iii)$ the c.p. maps $\phi: A\to B$ of the form $\phi(x)=\rho(x)u$, where $\rho\in A^*_{+}$ and $u\in B^{+}$, are admissible,
	
	$(iv)$ the composition, positive multiples, or finite sums of admissible c.p. maps are admissible.
	
\end{definition}

In the next definition \cite[Definition 4.1.1]{g}, $A_1$ denotes the closed unit ball of $A$.

\begin{definition} \label{oz}
Let $K\subseteq A_1$ be norm compact set and $\varepsilon>0$. A c.c.p. map $\varphi: A\to B$ is called approximately order 0 on $K$
within $\varepsilon$, or simply $(\varepsilon, K)$-order 0, if for each $\delta>0$ and  $a, b\in K$ with $\|ab\|<\delta$, there are 
elements $a^{'}, b^{'}\in A$ within $\frac{\delta}{2}$ of $a, b$, respectively, satisfying $\|\varphi(a^{'})\varphi(b^{'})\|< \varepsilon$. 
\end{definition}

\begin{definition} \label{nuclear}
A $\mathfrak A$-C*-algebra $A$ has module nuclear dimension (over $\mathfrak A$) at most $n$  if for every finite set $\mathcal F\ssubset A$ and $\varepsilon > 0$ there
exist a finite dimensional C*-algebra $F$, c.c.p. admissible map $\psi: A \to  F \otimes \mathfrak A$, a nonempty set $\mathcal G$ of c.p. admissible maps $\varphi: F\otimes \mathfrak A\to A$ such that $\|\varphi\circ\psi(a) - a\|< \varepsilon$, for $a\in \mathcal F$, $\varphi\in\mathcal G$, and $F$ decomposes into $n+1$ ideals $F = F^{(0)}\oplus \cdots\oplus F^{(n)}$  such that for each $i = 0, \cdots, n$,
any compact set $K_i\subseteq (F^{(i)}\otimes \mathcal A)_1$, and any $\varepsilon_i > 0$, there is a $\varphi\in\mathcal G$ whose restrictions $\varphi^{(i)}: F^{(i)}\otimes \mathfrak A\to A$ are $(\varepsilon_i, K_i)$-order 0.

In the above definition, if moreover the family $\mathcal G$ could be arranged to consist of c.p.c module maps $\varphi: F\otimes \mathfrak A\to A$, we say that module decomposition rank of $A$ (over $\mathfrak A$) is at most $n$.  

A $\mathfrak A$-C*-algebra $A$ has module nuclear dimension (resp., decomposition rank) $n$ over $\mathfrak A$, or simply, $\mathfrak A$-nuclear dimension (resp., $\mathfrak A$-decomposition rank) $n$, if $n$ is the least non-negative integer such that $A$ has module nuclear dimension (resp. decomposition rank) at most $n$. In this case, we write ${\rm dim}_{\rm nuc}^\mathfrak A(A)=n$ (resp., ${\rm dr}_{\mathfrak A}(A)=n$). If no such $n$ exists, we write ${\rm dim}_{\rm nuc}^\mathfrak A(A)=\infty$ (resp., ${\rm dr}_{\mathfrak A}(A)=\infty$). 
\end{definition}

The above definition of module nuclear dimension resembles that of relative nuclear dimension ${\rm nd}_\mathfrak A(A)$, defined by R. Gardner in his PhD thesis \cite[Definition 4.1.2]{g}, with the difference that Gardner does not have a module structure and so does not require the maps $\varphi$ and $\psi$ to be admissible maps. Since each approximate factoring via $F\otimes \mathfrak A$ in our definition is a factoring in the sense of Gardner (but not vice-versa), we calculate the least upper bound $n$ on a smaller set, that is,
$$ {\rm nd}_\mathfrak A(A)\leq {\rm dim}_{\rm nuc}^\mathfrak A(A).$$ 
When there is no module structure, since every  c.c.p. order map $\varphi: \mathbb M_n(\mathfrak A)\to A$ is $(\varepsilon, K)$-order 0, for each $\varepsilon>0$ and compact set $K\subseteq \mathbb M_n(\mathfrak A)_1$, we have ${\rm dim}_{\rm nuc}^\mathbb C(A)\leq {\rm dim}_{\rm nuc}(A).$ The reverse inequality now follows from \cite[Proposition 4.2.1]{g}, 
since  $${\rm dim}_{\rm nuc}(A)\leq {\rm nd}_\mathbb C(A)\leq {\rm dim}_{\rm nuc}^\mathbb C(A).$$

We use the standard notation $A_\infty:=\ell^\infty(A)/c_0(A)$, and $A_\omega:=\ell^\infty(A)/c_\omega(A)$, for free ultrafilter $\omega$ on $\mathbb N$. These have a canonical right module structure with compatible action. We need to compare admissible order 0 maps into $A$ with those into $A_\infty$. For this purpose, we need a version of Choi-Effros lifting theorem in the module setting. The original proof of this result \cite[Theorem 3.10]{ce} does not seem to be extendable to admissible maps, and here we adapt a different proof given in \cite[Theorem C.3]{bo}.

\begin{definition} \label{liftable}
	Let $E$ be an operator system and right $\mathfrak A$-module with compatible action, $J$ be a closed ideal and submodule of $B$. A c.c.p. admissible map $\varphi: E\to B/J$ is called liftable if there is a c.c.p. admissible map $\psi: E\to B$ with $q\circ \psi=\phi$, where $q: B\to B/J$ is the quotient map. 
\end{definition}

The proof of the next lemma is verbatim to that of \cite[Lemma C.2]{bo} (one just needs to keep track of inductive construction of maps and observe that in each step, all maps are admissible).

\begin{lemma} \label{closed}
If $E$ is separable, the set of liftable c.c.p. admissible maps: $E\to B/J$ is closed in point-norm topology. 
\end{lemma}

\begin{lemma} \label{ma}
	If all c.c.p. module maps: $A\to B/J$ are liftable, then so are all c.c.p. admissible maps: $A\to B/J$. 
\end{lemma}
\begin{proof}
	We need to check that  the constructions $(ii)$-$(iv)$ in Definition \ref{adm} preserves liftability. For $(ii)$, if $\varphi: A\to B/J$ is liftable with lift $\tilde\varphi:A\to B$, then for $u\in A$ and $v\in B$, the maps defined by $\varphi_1(x)=\varphi(uxu^*)$ and $\varphi_2(x)=q(v)\varphi(x)q(v)^*$ have lifts  $\tilde\varphi_1(x)=\tilde\varphi(uxu^*)$ and $\tilde\varphi_2(x)=v\tilde\varphi(x)v^*$. In $(iii)$, the map  $\varphi(x)=\rho(x)q(u)$, for $u\in B^{+}$ has lift  $\tilde\varphi(x)=\tilde\rho(x)u$, and finally, the composition, positive multiple and sum of liftable maps are clearly liftable.
\end{proof}

Now we could prove the analog of the celebrated Choi-Effros lifting theorem for admissible maps. We first recall the notion of $\mathfrak A$-nuclearity for objects and morphisms of the category of $\mathfrak A$-C*-algebras (c.f., \cite[Definitions 2.2, 2.6]{a2}).

\begin{definition}\label{nuc}
	A c.c.p. admissible map $\theta: A\to B$ is called $\mathfrak A$-nuclear if there are c.c.p. admissible maps $\varphi_n: A\to \M$ and $\psi_n: \M\to B$ such that $\psi_n\circ\varphi_n\to \theta$ in point-norm topology, that is,
	$$\|\psi_n\circ\varphi_n(a)-\theta(a)\|\to 0,$$
	for each $a\in A$.
	When id$_A$ is $\mathfrak A$-nuclear, we say that $A$ is $\mathfrak A$-nuclear.
\end{definition}

\begin{proposition}[Choi-Effros] \label{ce}
If $A$ and $\mathfrak A$ are separable, every $\mathfrak A$-nuclear c.c.p. admissible map: $A\to B/J$ is liftable. In particular, if $A$ is separable and $\mathfrak A$-nuclear, then every c.c.p. admissible map: $A\to B/J$ is liftable.
\end{proposition}
\begin{proof}
Since an $\mathfrak A$-nuclear c.c.p. admissible map approximately factors via $\mathbb M_n(\mathfrak A)$, by Lemma \ref{closed} it is enough to prove the assertion for the case $A=\mathbb M_n(\mathfrak A)$. By Lemma \ref{ma} we only need to show liftability for module maps. To a $\mathfrak A$-nuclear c.c.p. module map $\phi: \mathbb M_n(\mathfrak A)\to B/J$, one could associate a positive element $a\in \mathbb M_n(B/J)$ by \cite[Lemma 3]{a}, which in turn lifts to a positive element $b\in \mathbb M_n(B/J)$, and again using the above cited result, we get a lift $\tilde\phi: \mathbb M_n(\mathfrak A)\to B$, which is also a module map.
\end{proof}

\begin{remark} \label{mlift}
The above proof shows that if $A$ is separable and $\mathfrak A$-nuclear, then every c.c.p. module map: $A\to B/J$ is liftable to a {\it module map}: $A\to B$.
\end{remark}

\begin{corollary} \label{q}
	If $J$ is a closed ideal and submodule of $A$ such that $A/J$ is separable and $\mathfrak A$-nuclear, then ${\rm dim}_{\rm nuc}^\mathfrak A(A/J)\leq {\rm dim}_{\rm nuc}^\mathfrak A(A)$. 
\end{corollary}
\begin{proof}
	Let $d:={\rm dim}_{\rm nuc}^\mathfrak A(A)$. By Proposition \ref{ce}, the identity map on $A/J$ lifts to a c.c.p. admissible map $\sigma: A/J\to A$ with $q\circ \sigma={\rm id}$, where $q: A\to A/J$ is the quotient map. Given a finite set $\mathcal F\ssubset A/J$ and $\varepsilon>0$, pick a  $d$-decomposable admissible c.p.  approximation via $F\otimes\mathfrak A$ for $\sigma(\mathcal F)\subseteq A$ 
	within $\varepsilon$ with downward map $\psi$ and finite family $\mathcal G$ of upward maps. For any admissible $(\varepsilon, K)$-order 0 map $\varphi^{(i)}: F^{(i)}\otimes \mathfrak A\to A$, the admissible map $q\circ \varphi^{(i)}: F^{(i)}\otimes \mathfrak A\to A/J$ is $(\varepsilon, K)$-order 0, as $q$ is norm decreasing. For $\varphi\in\mathcal G$, 
	$$\|q\varphi\psi\sigma(a+J)-(a+J)\|=\|q\varphi\psi\sigma(a+J)-q\sigma(a+J)\|\leq\|\varphi\psi\sigma(a+J)-\sigma(a+J)\|<\varepsilon,$$
	for $b:=a+J\in F$.  Thus for $\tilde {\mathcal G}:=\{q\varphi: \varphi\in\mathcal G\}$, we get a  $d$-decomposable admissible c.p. approximation via $F\otimes\mathfrak A$ for $\mathcal F\subseteq A/J$ within
	$\varepsilon$, that is, ${\rm dim}_{\rm nuc}^\mathfrak A(A/J)\leq d$, as required.
\end{proof}

In general, we have the following estimates, as in the classical case \cite{wz}.

Next, we compare admissible order 0 maps into $A$ with those into $A_\infty$, as promised.

\begin{lemma} \label{seq}
Let $B$ be separable. If for every compact set $K\subseteq  B$ and $\varepsilon> 0$ there is a c.c.p. admissible $(\varepsilon, K)$-order 0 map $\varphi: B \to A$, then there is a sequence $(\varphi_n)$ of  such that $\phi(\cdot) = (\varphi_n(\cdot))$  defines
a c.c.p.  admissible order 0 map $\phi: B \to A_\infty$.

Conversely, if $B$ is moreover $\mathfrak A$-nuclear, for any given c.c.p.  order 0 admissible (module) map $\phi: B \to A_\infty$,  compact set $K \subseteq B$ and $\varepsilon> 0$, there is a c.c.p.  $(\varepsilon, K)$-order 0 admissible (module) map $\varphi:  B \to A$ and a sequence
$(\varphi_n)$ of  c.c.p. admissible (module) maps into $A$  such that $\phi(\cdot)=q\circ(\varphi_n(\cdot))$. 
\end{lemma}
\begin{proof}
The first statement is proved verbatim to \cite[Lemma 4.1.4]{g} (just by keeping track of admissibility). For the converse, we use the above module version of Choi-Effros lifting theorem (Proposition \ref{ce}) to find an admissible c.c.p.
lift $\tilde\phi: B \to \ell^\infty(A)$. Given compact subset $K\subseteq B_1$ and $\varepsilon> 0$ and $\delta>0$, take $a, b\in K$ with $\|ab\|<\delta$, and  choose orthogonal elements $a^{'}, b^{'}$ within $\frac{\delta}{2}$ of $a, b$ respectively. Since $\phi$  is
order 0,  $q(\tilde\phi(a^{'})\tilde\phi(b^{'})) =\phi(a^{'})\phi(b^{'}) = 0$,  where $q:\ell^\infty(A)\to A_\infty$ is the quotient
map. Choose a sequence
$(\varphi_n)$ of c.c.p. admissible maps into $A$ such that $\tilde\phi(\cdot)=(\varphi_n(\cdot))$, and observe that $\|\varphi_n(a^{'})\varphi_n(b^{'})\|\to 0$. Choose $N(a,b)\geq 1$ such that for $n\geq N(a,b)$,  $\|\varphi_n(a^{'})\varphi_n(b^{'})\|< \varepsilon.$ Since $K$ is compact, $N:= \sup_{a,b\in K} N(a,b)<\infty$, and $\varphi:= \varphi_{N}$ is an admissible c.c.p. $(\varepsilon, K)$-order 0 map into $A$. Finally, we have $\phi(\cdot)=q\circ\tilde\phi(\cdot)=q\circ(\varphi_n(\cdot))$.

Finally, when $\phi$ is a module map, by Remark \ref{mlift}, we may choose $\tilde\phi$ and each $\varphi_n$ to be a c.c.p. module map. 
\end{proof}
 
In the next result, $\iota: A\to A_\infty$ is the canonical inclusion (which is a module map). The proof uses Lemma \ref{seq} and is omitted as it goes verbatim to that of \cite[Proposition 4.1.6]{g}.  

\begin{corollary} \label{rel} 
	Let $\mathfrak A$ be separable. Then for  each $n\geq 1$, the following are equivalent.

	$(i)$ ${\rm dim}_{\rm nuc}^\mathfrak A(A)\leq n,$
	
	$(ii)$ there exists a net $(F_i)$ of finite dimensional
	C*-algebras, and for each $i$ there exist admissible c.c.p. maps $\psi_i: A\to F_i\otimes \mathfrak A$  and admissible c.p. maps $\varphi_i: F_i\otimes \mathfrak A\to A_\infty$
	such that $\varphi_i\circ\psi_i\to \iota$ in point-norm, and for each $i$, we have a decomposition $F_i=F_i^{(0)}\oplus\cdots\oplus F_i^{(n)}$ into $n + 1$ ideals such that the restriction of $\varphi_i$ to each $F_i^{(k)}$ is an admissible c.c.p. order 0 map, for $k=0,\cdots, n$.
\end{corollary}

Note that finiteness of module nuclear dimension does not imply nuclearity of $A$, for instance, we always have ${\rm dim}_{\rm nuc}^A(A)=0$. However, we show that ${\rm dim}_{\rm nuc}^\mathfrak A (A)<\infty$ implies that $A$ is $\mathfrak A$-nuclear.     

\begin{theorem} \label{nu}
$(i)$ If $A$ has finite $\mathfrak A$-module dimension, then $A$ is $\mathfrak A$-nuclear,

$(ii)$ If $A$ has finite $\mathfrak A$-module dimension and $\mathfrak A$ is nuclear, then $A$ is nuclear.
\end{theorem}
\begin{proof}
	$(i)$ We have  ${\rm nd}_\mathfrak A(A)\leq {\rm dim}_{\rm nuc}^\mathfrak A(A)<\infty$, so the assertion follows from \cite[Proposition 4.2.3]{g}.

$(ii)$ If ${\rm dim}_{\rm nuc}^\mathfrak A(A)=d$, then in the approximate decomposition of ${\rm id}_A$, the c.p. maps going up are sums of $d+1$ order zero maps, and so are of norm at most $d+1$. If we scale these maps with $\lambda:=1/(d+1)$, we get c.c.p. admissible approximate factorization of $\lambda{\rm id}_A$, which means that $\lambda{\rm id}_A$ is $\mathfrak A$-nuclear, and so is ${\rm id}_A$.
\end{proof}

The next result gives an upper bound of the nuclear dimension of $A$ with respect to the $\mathfrak A$-module nuclear dimension of $A$ and nuclear dimension of $\mathfrak A$. Again the proof uses Lemma \ref{seq} and goes verbatim to that of \cite[Proposition 4.2.1]{g}. As usual, we use the abbreviation $n^{+}:=n+1$.

\begin{corollary} \label{ub} 
	If $\mathfrak A$ is separable and ${\rm dim}_{\rm nuc}(\mathfrak A)$ and  ${\rm dim}_{\rm nuc}^\mathfrak A (A)$ are both finite,  then  so is ${\rm dim}_{\rm nuc}(A)$ and we have $${\rm dim}_{\rm nuc}^{+}(A)\leq {\rm dim}_{\rm nuc}^{\mathfrak A +}(A){\rm dim}_{\rm nuc}^{+}(\mathfrak A).$$
	In particular, if $\mathfrak A$ is an AF-algebra, ${\rm dim}_{\rm nuc}(A)\leq {\rm dim}_{\rm nuc}^\mathfrak A(A).$ 
\end{corollary}	

\begin{corollary} \label{c} 
	${\rm dim}_{\rm nuc}^\mathbb C(A)= {\rm dim}_{\rm nuc}(A).$ 
\end{corollary}

\begin{remark} \label{es}
$(i)$ By the above theorem, we may remove the condition of $\mathfrak A$-nuclearity from Corollary \ref{q}, that is, ${\rm dim}_{\rm nuc}^\mathfrak A(A/J)\leq {\rm dim}_{\rm nuc}^\mathfrak A(A)$.

$(ii)$ The fact that for a hereditary C*-subalgebra $B\leq A$, the nuclear dimension of $B$ does not exceed that of $A$ \cite[Proposition 2.5]{w} uses the existence of supporting homomorphisms for c.c.p. order 0 maps \cite[Proposition 3.2]{w}, but as we would see later (c.f., Lemma \ref{homo}), this is not the case for general admissible maps (and we need the upward maps to be module map). In particular, for a closed ideal and submodule $J\unlhd A$, we may only conclude that ${\rm dr}_\mathfrak A(J)\leq {\rm dr}_\mathfrak A(A)$.

$(iii)$ In the classical case, one has $ {\rm dim}_{\rm nuc}(A)\leq {\rm dim}_{\rm nuc}(J)+{\rm dim}_{\rm nuc}(A/J)+1$ \cite[Proposition 2.9]{w}, but the proof uses the fact that the cone $CF$ of a finite dimensional C*-algebra $F$ is projective. This does not hold for $C(F\otimes\mathfrak A)$, in particular, we do not have the module version of this inequality, except in very special cases (see Example \ref{eseq} in Section \ref{5}).  
\end{remark}

Next, we give characterizations of $\mathfrak A$-C*-modules with zero $\mathfrak A$-module nuclear dimension. Part $(ii)$ of the next definition extends \cite[Definition 2.1.1]{bk}.
 
\begin{definition} \label{ind}
	$(i)$ An inductive system of $\mathfrak A$-C*-algebras is a
	sequence $(A_n)$ of $\mathfrak A$-C*-algebras, with admissible $*$-homomorphisms $\phi_{nm}: A_n \to  A_m$, for $n < m$, satisfying $\phi_{mk}\circ\phi_{nm}=\phi_{nk}$, for $n<m<k$. 
	
	$(ii)$ A generalized inductive system of $\mathfrak A$-C*-algebras is a
sequence $(A_n)$ of $\mathfrak A$-C*-algebras, with admissible maps $\phi_{nm}: A_n \to  A_m$, for $n < m$, such
that for all $n\geq 1$ and all $x, y\in A_n$, $\lambda\in\mathbb C$, and $\varepsilon> 0$, there is a positive integer $M>n$ such that for all
$M \leq m < k$,

(1) $\|\phi_{nk}(x)-\phi_{mk}\circ\phi_{nm}(x)\|<\varepsilon,$

(2) 
$\|\phi_{mk}\big(\phi_{nm}(x)+\phi_{nm}(y)\big)-\big(\phi_{nk}(x)+\phi_{nk}(y)\big)\|<\varepsilon,$

(3) $\|\phi_{mk}\big(\lambda\phi_{nm}(x)\big)-\lambda\phi_{nk}(x)\|<\varepsilon,$

(4) $\|\phi_{mk}\big(\phi_{nm}(x)^*\big)-\phi_{nk}(x)^*\|<\varepsilon,$

(5) $\|\phi_{mk}\big(\phi_{nm}(x)\phi_{nm}(y)\big)-\phi_{nk}(x)\phi_{nk}(y)\|<\varepsilon,$

(6) 
$\sup_r \|\phi_{nr}(x)\|< \infty$.
\end{definition}

Note that in $(ii)$ the maps are not linear or c.p., so the condition of admissibility here is an extension of the notion defined in Definition \ref{adm}, simply by removing the linearity and c.p. assumptions in parts $(i)$-$(iv)$. When the connecting maps are linear and c.p., all conditions except (1) and (5) are automatic. 

An alternative construction could be given using limit algebra \cite[2.1.3]{bk}: define $\varphi_k: A_k\to \prod_{n} A_n/\bigoplus_n A_n$, with $\varphi_k(x)=\phi_{kn}(x)$, for $n>k$, and zero otherwise. We say that the generalized inductive system is {\it regular}, if  connecting maps $\phi_{nm}: A_n\to A_m$ are linear and admissible c.c.p. and  induced maps $\varphi_n: A_n\to \lim \limits_{\to}A_n$ are $*$-homomorphisms (this is not a standard terminology and used only in the next theorem for the sake of brevity). 

For an (a generalized) inductive systems $(A_n, \phi_{nm})$ we denote the corresponding (generalized) inductive limit with $\lim \limits_{\to}A_n$. In the limit algebra interpretation, each map $\varphi_k$ is indeed into the generalized inductive limit $\lim \limits_{\to}A_n$.

\begin{definition} \label{reg}
	$(i)$ The inductive limit of an inductive system of $\mathfrak A$-C*-algebras of the form $A_n=F_n\otimes \mathfrak A$, where $F_n$ is a finite dimensional C*-algebra, is called an $\mathfrak A$-AF algebra.
	
	$(ii)$ The generalized inductive limit of a generalised inductive system of $\mathfrak A$-C*-algebras of the form $A_n=F_n\otimes \mathfrak A$, where $F_n$ is a finite dimensional C*-algebra, is called an $\mathfrak A$-NF algebra. An $\mathfrak A$-NF algebra is called regular, if the corresponding system is regular in the above sense.
\end{definition}

The main argument of the proof of the next result adapts that of \cite[Proposition 4.3.3]{g}, but note that here we prove the result under much weaker condition of $\mathfrak A$-nuclearity of $A$ (where as in \cite{g}, $A$ is assumed to be nuclear).
 
\begin{theorem} \label{af}
Assume that $A$ is separable and $\mathfrak A$-nuclear.	Consider the following assertions: 
	
	$(i)$ $A$ is an $\mathfrak A$-AF algebra,
	
	$(ii)$ $A$ is a regular $\mathfrak A$-NF algebra,
	
	$(iii)$  ${\rm dim}_{\rm nuc}^\mathfrak A(A)=0$.
	
\noindent	Then we have, $(i)\Rightarrow(ii)\Rightarrow(iii).$
\end{theorem}
\begin{proof}
$(i)\Rightarrow(ii)$. This is immediate as for inductive systems all the maps $\phi_{nm}$ are assumed to be $*$-homomorphisms. 

$(ii)\Rightarrow(iii)$. Let $A:=\lim \limits_{\to}(F_n\otimes \mathfrak A)$ be a regular $\mathfrak A$-NF algebra with connecting $*$-homomorphism $\varphi_n: A_n\to A$. Given a finite subset $\mathcal F\ssubset A$ and $\varepsilon > 0$, there is $N\geq 1$, and for each $a\in\mathcal F$ there exists $b:=b(a)\in F_N\otimes\mathfrak A$ with $\|a -\varphi_N(b)\|<4\varepsilon.$ Let $\mathcal G:=\{b(a): a\in \mathcal F\}\subseteq F_N\otimes \mathfrak A$. By Proposition \ref{ce}, we may lift the module map embedding $$\iota: A\hookrightarrow \prod_{n}(F_n\otimes \mathfrak A)/\bigoplus_n(F_n\otimes \mathfrak A)$$ to an admissible c.c.p. map $\tilde\iota: A\hookrightarrow \prod_{n}(F_n\otimes \mathfrak A)$. Let us consider the projection $\pi_m: \prod_{n}(F_n\otimes \mathfrak A)\to F_m\otimes \mathfrak A$ onto the $m^{\rm th}$-component. Then $\tilde\iota\circ\varphi_N$ lifts $\varphi_N$, thus we may find some $M_0\geq 1$ such that for all
$m\geq M_0$ and $b\in\mathcal G$, $$\|\phi_{nm}(b) - \pi_m(\varphi_N (b))\|< 4\varepsilon.$$ By condition $(1)$ in Definition \ref{ind}, there is $M_1\geq 1$  such that for all $M_1\leq m < k$, 
$$\|\phi_{nk}(b) - \phi_{mk}\big(\phi_{nm}(b)\big)\|< 4\varepsilon,$$ 
for $b\in\mathcal G$, and in particular,
$$\|\varphi_N(b) - \varphi_{m}\big(\phi_{Nm}(b)\big)\|< 4\varepsilon,$$ 
for $m\geq M_1$ and $b\in\mathcal G$. Put $M = \max \{M_0,M_1\}$, and observe that 
\begin{align*}
\|\varphi_M\circ\pi_M(a)-a\|&\leq \|\varphi_M\circ\pi_M (\varphi_N(b(a)) - \varphi_N(b(a))\|+2 \|a-\varphi_N(b(a))\|\\ &\leq \|\varphi_M\circ\pi_M (\varphi_N(b(a)) - \varphi_M(\phi_{NM}(b(a)))\|\\&+\|\varphi_M(\phi_{NM}(b(a)))-\varphi_N(b(a))\|+2 \|a-\varphi_N(b(a))\|\\&<\frac{\varepsilon}{4}+\frac{\varepsilon}{4}+2(\frac{\varepsilon}{4})=\varepsilon.
\end{align*}
This means that ${\rm id}_A$ admits an approximate admissible c.p. factoring via $F_n\otimes\mathfrak A$ with downward map $\pi_n$ which is admissible c.c.p., and upward map $\varphi_n$, which is admissible c.p. and order 0 (indeed a $*$-homomorphism). This means that $\mathfrak A$-module nuclear dimension of $A$ is zero, as required.  
\end{proof}

Since conditions $(i)$ and $(ii)$ are not equivalent, we could not expect the equivalence of $(i)$ and $(iii)$ as in the classical case. In Section \ref{5} we provide classes of examples of $\mathfrak A$ for which the last two conditions are equivalent (see, Example \ref{equi}).  

Next we show that $\mathfrak A$-C*-algebras of finite module decomposition rank are $\mathfrak A$-quasidiagonal. Let us first recall the definition of $\mathfrak A$-quasidiagonality from \cite{a3}.

\begin{definition} \label{qd}
	A $C^*$-module $A$ is called $\mathfrak A$-{\it quasidiagonal} (briefly, $\mathfrak A$-QD) if there exists a net of admissible c.c.p. maps $\phi_n: A \to \mathbb M_{k(n)}(\mathfrak A)$ which are approximately multiplicative and approximately isometric, i.e.,
	$\|\phi_n(ab)-\phi_n(a)\phi_n(b)\|\to 0$ and $\|\phi_n(a)\|\to \|a\|,$ as $n\to\infty$,
	for all $a, b\in A$; or equivalently, if for each finite set $\mathfrak F\subseteq A$ and $\varepsilon>0$, there is a positive integer $k$ and a c.c.p. admissible map $\phi: A \to \mathbb M_{k}(\mathfrak A)$ satisfying $\|\phi(ab)-\phi(a)\phi(b)\|<\varepsilon$ and $\|\phi(a)\|> \|a\|-\varepsilon,$
	for all $a, b\in \mathfrak F$.
\end{definition}

In the next two lemmas and the theorem after those, we further assume that the action of $\mathfrak A$ on $A$ satisfies 
$$(a\cdot \alpha)^*=a^*\cdot \alpha^*,\ \ (a\in A, \alpha\in\mathfrak A).$$

The next lemma extends part of \cite[Proposition 3.2]{w}.

\begin{lemma} \label{homo}
Let $\mathfrak A$ be unital and $A$ be separable. Let $\varphi: \mathbb M_n(\mathfrak A)\to A$ be a c.c.p. order 0 module map, then there is a unique  $*$-homomorphism and module map $\pi_\varphi: C_0(0,1]\otimes \mathbb M_n(\mathfrak A)\to C^*({\rm Im}(\varphi))\subseteq A$ satisfying $\pi_\varphi({\rm id}\otimes x)=\varphi(x)$, for $x\in\mathbb M_n(\mathfrak A)$, where ${\rm id}: (0,1]\to (0, 1]$ is the identity function. 
\end{lemma}
\begin{proof}
	Since $A$ is separable, there is a faithful non-degenerate representation $A\subseteq \mathbb B(H\otimes\mathfrak A)$, for some Hilbert space $H$ \cite[Lemma 6.4]{lan}. We may assume that $C^*({\rm Im}(\varphi))=A$. Let $h:=\varphi(1_n\otimes 1_\mathfrak A)$, where $1_n$ is the identity of $\mathbb M_n(\mathbb C)$. Let $p$ be the support projection of $h$ in the von Neumann algebra $\mathbb B(H\otimes \mathfrak A)^{**}$. Then, given a self-adjoint element $x\in \mathbb M_n(\mathfrak A)$ with $\|x\|\leq 1$, $\varphi(x)\leq \varphi(1_n\otimes 1_\mathfrak A) \leq p$, i.e., $p\varphi(x)=\varphi(x)p=\varphi(x)$, inside $\mathbb B(H\otimes \mathfrak A)^{**}$. Thus $pa=ap=a$, for each $a\in A$. This along with non-degeneracy of inclusion $A\subseteq \mathbb B(H\otimes \mathfrak A)$, says that $p= \mathbbm 1$, the identity of $\mathbb B(H\otimes \mathfrak A)$, which is also the identity of $\mathbb B(H\otimes \mathfrak A)^{**}$. Consider the sequence of c.c.p. maps,
	$$\sigma_k: \mathbb M_n(\mathfrak A)\to  \mathbb B(H\otimes\mathfrak A); \ x\mapsto (h+1/k\mathbbm 1)^{\frac{-1}{2}}\varphi(x) (h+1/k\mathbbm 1)^{\frac{-1}{2}}.$$
	It is well known that the sequence $\{(h+1/k\mathbbm 1)^{\frac{-1}{2}}\}$ SOT-converges to its support projection $p=\mathbbm 1$ inside the von Neumann algebra $\mathbb B(H\otimes \mathfrak A)^{**}$ (c.f., \cite[5.4.7]{s}). Now the same holds for the norm-bounded sequence $\{(h+1/k\mathbbm 1)^{\frac{-1}{2}}h(h+1/k\mathbbm 1)^{\frac{-1}{2}}\}$. For $0\leq x\leq 1_n\otimes 1_\mathfrak A$ in $\mathbb M_n(\mathfrak A)$, we have $0\leq \varphi(x)\leq h$, and so the norm-bounded increasing sequence  $\{(h+1/k\mathbbm 1)^{\frac{-1}{2}}\varphi(x)(h+1/k\mathbbm 1)^{\frac{-1}{2}}\}$ SOT-converges to an element $\sigma(x)\in\mathbb B(H\otimes \mathfrak A)^{**}$. Extending linearly, we get a map $\sigma: \mathbb M_n(\mathfrak A)\to  \mathbb B(H\otimes\mathfrak A)^{**}$, satisfying,
	$$h^{\frac{1}{2}}\sigma(x)h^{\frac{1}{2}}=p\varphi(x)p=\varphi(x)\ \ (x\in \mathbb M_n(\mathfrak{A})).$$ We claim that $\sigma$ is a $*$-homomorphism and a module map. To prove the claim, let us first observe that,
	$$h\in {\rm Im}(\sigma)^{'}\subseteq \mathbb B(H\otimes \mathfrak A)^{**}.$$ 
	Given a self-adjoint element  $t\in\mathbb M_n(\mathbb{C})$, $\alpha\in\mathfrak{A}$, and $x=t\otimes\alpha$, we may write $x=\sum_{j=1}^{n} \lambda_je_{jj}\otimes \alpha$, for the matrix unit $\{e_{ij}\}\subseteq \mathbb M_n(\mathbb C)$. Since $\varphi$ is a module map, 
	\begin{align*}
		h\varphi(x)&=\varphi(1_n\otimes 1_\mathfrak{A})\varphi(x)\\&=
		\sum_{i,j} \lambda_j\varphi(e_{ii}\otimes 1_\mathfrak A)\varphi(e_{jj}\otimes 1_\mathfrak A)\cdot \alpha\\&=\sum_{j} \lambda_j\varphi(e_{jj}\otimes 1_\mathfrak A)^2\cdot \alpha,
	\end{align*}
where the last equality follows from the fact that $\varphi$ is order 0. Now since $(a\cdot\alpha)^*=a^*\cdot\alpha^*,$ we have,
$$(a\cdot\alpha)b=(b^*(a\cdot\alpha)^*)^*=(b^*(a^*\cdot\alpha^*))^*=((b^*a^*)\cdot \alpha^*)^*=(ab)\cdot\alpha,$$
for $a,b\in A$ and $\alpha\in \mathfrak{A}.$ Therefore,
 \begin{align*}
 	\varphi(x)h&=\varphi(x)\varphi(1_n\otimes 1_\mathfrak{A})\\&=
 	\sum_{i,j} \lambda_j(\varphi(e_{jj}\otimes 1_\mathfrak A)\cdot \alpha)\varphi(e_{ii}\otimes 1_\mathfrak A)\\&=	\sum_{i,j} \lambda_j\varphi(e_{jj}\otimes 1_\mathfrak A)\varphi(e_{ii}\otimes 1_\mathfrak A)\cdot \alpha\\&=\sum_{j} \lambda_j\varphi(e_{jj}\otimes 1_\mathfrak A)^2\cdot \alpha,
 \end{align*}  
hence, $	\varphi(x)h=h\varphi(x)$, and so $h\sigma(x)=\sigma(x)h$, for elements $x$ of the above form, and so for each $x\in\mathbb M_n(\mathfrak A)$, proving the claim. Next, let us observe that $\sigma$ is a module map and a $*$-homomorphism. For $x\in \mathbb M_n(\mathfrak{A})$ and $\alpha\in\mathfrak{A}$, since $h$ commutes with range of $\varphi$,
\begin{align*}
h^{\frac{1}{2}}(\sigma(x\cdot\alpha)-\sigma(x)\cdot\alpha)h^{\frac{1}{2}}&= h^{\frac{1}{2}}\sigma(x\cdot\alpha)h^{\frac{1}{2}}-h^{\frac{1}{2}}(\sigma(x)\cdot\alpha)h^{\frac{1}{2}}\\&=\varphi(x\cdot\alpha)-h^{\frac{1}{2}}(\sigma(x)\cdot\alpha)h^{\frac{1}{2}} \\&=\varphi(x\cdot\alpha)-\lim_kh^{\frac{1}{2}}((h+1/k\mathbbm 1)^{-1}\varphi(x)\cdot\alpha)h^{\frac{1}{2}} \\&=\varphi(x\cdot\alpha)-\lim_kh^{\frac{1}{2}}(h+1/k\mathbbm 1)^{-1}\varphi(x\cdot\alpha)h^{\frac{1}{2}}	\\&=\varphi(x\cdot\alpha)-\lim_kh(h+1/k\mathbbm 1)^{-1}\varphi(x\cdot\alpha)\\&=\varphi(x\cdot\alpha)-\mathbbm 1\varphi(x\cdot\alpha)=0.			
\end{align*}
Since the support projection of $h$ is $\mathbbm 1$, we have $hyh=0$ iff $y=0$, for $y\in\mathbb B(H\otimes \mathfrak{A})^{**}.$ In particular, multiplying both sides of above equalities with $h^{\frac{1}{2}}$, we get $\sigma(x\cdot\alpha)=\sigma(x)\cdot\alpha$, as required. Next, take a projection $e_1\in\mathbb M_n(\mathfrak A)$ and put $e:=e_1\otimes 1_\mathfrak{A}$ and $1-e=(1-e_1)\otimes 1_\mathfrak{A}$.  In particular,
$$h\sigma(e)\sigma(1-e)h=h^{\frac{1}{2}}\sigma(e)h^{\frac{1}{2}}h^{\frac{1}{2}}\sigma(1-e)h^{\frac{1}{2}}=\varphi(a)\varphi(1-e)=0,$$
implies $\sigma(e)\sigma(1-e)=0$, therefore,
$$\|\sigma(ex)-\sigma(e)\sigma(x)\|\leq \|\sigma(ee^*)-\sigma(e)\sigma(e)^*\|\|x\|=0,$$
 since $\sigma$ is c.c.p., that is, $\sigma(ex)=\sigma(e)\sigma(x)$, and similarly,  $\sigma(xe)=\sigma(e)\sigma(x)$, for each $x\in\mathbb M_n(\mathfrak{A})$. Now let $x_1:=e_1\otimes \alpha=e\cdot\alpha$, for some $\alpha\in\mathfrak{A}$, and observe that for each $x$,
 \begin{align*}
 \sigma(xx_1)&=\sigma(x(e\cdot\alpha))=\sigma(xe\cdot\alpha)=\sigma(xe)\cdot\alpha\\&=
 \sigma(x)\sigma(e)\cdot\alpha= \sigma(x)\sigma(e\cdot\alpha)= \sigma(x)\sigma(x_1),			
 \end{align*}
 since $\sigma$ is a module map. Since linear combinations of elements of the form $x_1$ above generates $\mathbb M_n(\mathfrak{A})$, it follows that $\sigma$ is a homomorphism. Finally, $\sigma$ preserves adjoints as it is c.c.p.  
 
 Since the support projection of $h$ is $\mathbbm 1$, $h$ is invertible in $\mathbb B(H\otimes \mathfrak{A})^{**}$, and so $0\notin {\rm sp}_A(h)$, as the spectrum does not change going to a C*-subspace. In particular, $h$ is invertible in the minimal unitization of $A$. This in turn implies that ${\rm Im}(\sigma)=C^*({\rm Im}(\varphi))= A$.  
 
 Next, $C^*(h)\subseteq A$ is $*$-isomorphic to $C_0({\rm sp}(h))$, which in turn is a quotient of $C_0(0, 1]$, as $\|h\|\leq 1.$ Let $\pi:C_0(0,1]\twoheadrightarrow
 C^*(h)\subseteq A$ be the corresponding surjection. Since the range of $\pi$ commutes with the range of $\sigma: \mathbb M_n(\mathfrak{A})\to A$, we get a $*$-homomorphism $\pi_\varphi:=\pi\otimes\sigma: C_0(0,1]\otimes\mathbb M_n(\mathfrak{A})\to A$, satisfying $\pi\otimes\sigma(f\otimes x)=\pi(f)\otimes \sigma(x),$ for $x\in\mathbb M_n(\mathfrak{A})$. In particular, identifying $A=C^{*}({\rm Im}(\varphi))$ with $C^*(h)\otimes {\rm Im}(\sigma)$, for $f={\rm id}$, we get $\pi_\varphi({\rm id}\otimes x)=h\otimes\sigma(x)=\varphi(x),$ as required.      
\end{proof}
\begin{lemma}\label{key}
	Let $A$ be separable and $\mathfrak A$ be unital and simple. Then the following are equivalent:
	
	$(i)$ ${\rm dr}_\mathfrak A(A)\leq n<\infty,$
	
	$(ii)$ for any finite subset $\mathcal F\ssubset A$ and $\varepsilon > 0$, there is a finite dimensional C*-algebra $F = F^{(0)}\oplus \cdots\oplus F^{(n)}$ and admissible c.c.p. map $\psi:  A\to F\otimes\mathfrak A$ and c.c.p. module map $\varphi: F\otimes\mathfrak A\to A$ satisfying,
	
	{\rm (1)} $\|\varphi\circ\psi(a)-a\|<\varepsilon$,
	
	{\rm (2)} $\|\psi(ab)-\psi(a)\psi(b)\|<\varepsilon$,
	
	{\rm (3)} the restriction $\varphi^{(i)}$ of $\varphi$ to $F^{(i)}\otimes\mathfrak A$ is order 0,
	
	\noindent for $a,b\in \mathcal F$ and $0\leq i\leq n$. 
\end{lemma}
\begin{proof}
	We only need to prove $(i)\Rightarrow(ii)$, as the other direction is trivial. We adapt the proof of \cite[17.4.3]{s}. Since $\psi$ is c.c.p. it satisfies (cf., \cite[7.3.5]{s}) the inequality
	$$\|\psi(ab)-\psi(a)\psi(b)\|<\|\psi(aa^*)-\psi(a)\psi(a^*)\|^{\frac{1}{2}}\|b\|,$$
	in particular, as we may always assume that $\mathcal F$ consists of positive elements of norm at most 1, (2) is equivalent to 
	
	(2)$^{'}$ $\|\psi(a^2)-\psi(a)^2\|<\varepsilon\  \ (a\in\mathcal F).$
	
\noindent	Let $\delta<\varepsilon^3/(\varepsilon^2+40(n+1))$. Choose finite dimensional C*-algebra $F^{'} = F^{'(0)}\oplus \cdots\oplus F^{'(n)}$, an admissible c.c.p. map $\psi^{'}: A\to F^{'}\otimes\mathfrak A$, and a c.c.p. module map $\varphi^{'}: F^{'}\otimes\mathfrak A\to A$ satisfying, 
$$\|\psi^{'}\circ\varphi^{'}(a)-a\|<\delta\ \  (a\in\mathcal F\cup\mathcal F^2),$$
\noindent with $\varphi^{'}_i:=\varphi^{'}|_{F^{'(i)}\otimes\mathfrak A}$ order 0, for $0\leq i\leq n$. We may further assume that $F^{'(i)}=\mathbb M_{r_i}(\mathbb C)$. Let $\pi_{\varphi^{'}_i}: \mathbb M_{r_i}(\mathfrak A)\to A$ be the corresponding $*$-homomorphism and module map as in Lemma \ref{homo}. Then $\pi_{\varphi^{'}_i}$ is injective (and so an isometry) as $\mathfrak A$ is simple. In particular, 
$$\|\varphi^{'}_i(x)\|=\|\pi_{\varphi^{'}_i}({\rm id}\otimes x)\|=\|x\|\|h_i\|,$$ for $x\in  \mathbb M_{r_i}(\mathfrak A)$ and $1\leq i\leq n$, where $h_i:=\varphi^{'}_i(1_{r_i}\otimes 1_\mathfrak A)$. Let $\psi^{'}_i$ be the restriction of $\psi^{'}$ to the inverse image of $F^{'(i)}\otimes \mathfrak A$ under $\psi^{'}$. Given $\varepsilon>0$, let $I:=\{1,\cdots, n\}$ and $I(\varepsilon)$ consist of those $i\in I$ satisfying $\|\psi^{'}_i(a)^2-\psi^{'}_i(a^2)\|<\varepsilon/(n+1),$ for each  $a\in \mathcal F$. For  $i\in I\backslash I(\varepsilon)$, there is $a_i\in\mathcal F$ such that,
\begin{align*}
	\varepsilon^2\|h_i\|&\leq \|\varphi^{'}_i\big(\psi^{'}_i(a_i)^2-\psi^{'}_i(a_i^2)\big)\|\leq \|\varphi^{'}\big(\psi^{'}(a_i)^2-\psi^{'}(a_i^2)\big)\|\\&\leq \|\varphi^{'}\circ\psi^{'}(a_i)^2-a_i^2\|+\|a_i^2-\varphi^{'}\circ\psi^{'}(a_i^2)\|\\&\leq \|\big(\varphi^{'}\circ\psi^{'}(a_i)-a_i\big)\big(\varphi^{'}\circ\psi^{'}(a_i)+a_i\big)\|+\|\varphi^{'}\circ\psi^{'}(a_i)a_i-a_i^2\|\\&+\|a_i^2-a_i\varphi^{'}\circ\psi^{'}(a_i)\|+\|a_i^2-\varphi^{'}\circ\psi^{'}(a_i^2)\|\\&\leq 4\|\varphi^{'}\circ\psi^{'}(a_i)-a_i\|+\|a_i^2-\varphi^{'}\circ\psi^{'}(a_i^2)\|\leq 5\delta,
\end{align*}
as $\|\varphi^{'}\circ\psi^{'}(a_i)\|\leq \|a_i\|\leq 1;$ i.e., $\|h_i\|\leq 5\delta/\varepsilon^2,$ for $i\in I\backslash I(\varepsilon)$. Put $F:=\bigoplus_{i\in I(\varepsilon)} \mathbb M_{r_i}(\mathbb C)$, and define admissible c.c.p. maps $\varphi:= \varphi^{'}|_{F\otimes \mathfrak A}$ and 
$$\psi: A\to F\otimes \mathfrak A;\ a\mapsto (1_F\otimes 1_\mathfrak A)\psi^{'}(a)(1_F\otimes 1_\mathfrak A),$$
then for $p:=1_F\otimes 1_\mathfrak A$, $1-p:=(1_{F^{'}}-1_F)\otimes 1_\mathfrak A$, and $q:=\varphi^{'}(p)$, we have 
$$\|q\|\leq\sum_{i\in I(\varepsilon)}\|\varphi^{'}(1_{r_i}\otimes 1_\mathfrak A)\|=\sum_{i\in I(\varepsilon)}\|h_i\|\leq 5(n+1)\delta/\varepsilon^2,$$ and 
\begin{align*}
	\|\psi(a)^2-\psi(a^2)\|&=\|p\psi^{'}(a)p\psi^{'}(a)p-p\psi^{'}(a^2)p\|\\&\leq \|(p\psi^{'}(a))^2-p\psi^{'}(a^2)\|\\&
	\leq \sum_{i\in I(\varepsilon)}\|\psi^{'}_i(a)^2-\psi^{'}_i(a^2)\|<\varepsilon,
\end{align*} 
for each $a\in\mathcal F$. Also, for $b:=(1-p)\psi^{'}(a)$ and $c:=\psi^{'}(a)(1-p)$, we have $\|b\|\leq \|\varphi^{'}(a)\|\leq\|a\|\leq 1$, and the same for $c$, therefore, 
\begin{align*}
	\|\varphi\circ\psi(a)-a\|&=\|\varphi\big(p\psi^{'}(a)p\big)-a\|\\&\leq\|\varphi^{'}\big(\psi^{'}(a)\big)-a\|+\|\varphi^{'}\big((1-p)\psi^{'}(a)p\big)\|+\|\varphi^{'}\big(p\psi^{'}(a)(1-p)\big)\|\\&\leq\delta+\|\varphi^{'}(bp)\|+|\varphi^{'}(pc)\|\\&\leq\delta+\|\varphi^{'}(p^2)-\varphi^{'}(p)^2\|(\|b\|+\|c\|) \\&\leq\delta+4\|\varphi^{'}(p)\|(1+\|\varphi^{'}(p)\|)\\&\leq\delta+8\|q\|\leq \delta(1+40(n+1)/\varepsilon^2)<\varepsilon,
\end{align*} 
for each $a\in\mathcal F$, as required. 
\end{proof}

When $A$ is separable, one could always faithfully represent $A$ in $\mathbb B(H\otimes \mathfrak A)$, for some separable Hilbert space $H$ \cite[Lemma 6.4]{lan}. In this case, this is equivalent to the existence of a sequence $(p_n)$ of finite rank projections in $\mathbb K(H\otimes \mathfrak A)$, SOT-increasing to the identity operator in $\mathbb B(H\otimes \mathfrak A)$, and approximately commuting with $A$ (c.f., \cite[Lemma 3.4]{a3}). Here, being finite rank simply means that the image of $p_n$ is of the form $H_n\otimes\mathfrak A$, for some finite dimensional (closed) subspace $H_n\leq H$.  

In the next result, we assume that the action satisfies the condition stated before Lemma \ref{homo}.   

\begin{theorem}
	Let $\mathfrak A$ be unital and simple. Then every  separable $\mathfrak A$-C*-algebra $A$ with finite $\mathfrak A$-module decomposition rank is $\mathfrak A$-quasidiagonal.
\end{theorem}
\begin{proof}
If $A$ is a separable $\mathfrak A$-C*-algebra with	${\rm dr}_\mathfrak A(A)\leq n<\infty,$ then  for any finite subset $\mathcal F\ssubset A$ and $\varepsilon > 0$, there is a finite dimensional C*-algebra $F = F^{(0)}\oplus \cdots\oplus F^{(n)}$ and admissible c.c.p. map $\psi:  A\to F\otimes\mathfrak A$ and c.c.p. module map $\varphi: F\otimes\mathfrak A\to A$ satisfying conditions (1)-(3) in part $(ii)$ of by Lemma \ref{key}. In particular,
$$\|\psi(a)\|\geq \|\varphi(\psi(a))\|>\|a\|-\varepsilon,$$
for each $a\in\mathcal F$. For the indices  $i:=(\varepsilon, \mathcal F)$ with natural partial order, if $\psi_i$ is the above admissible c.c.p. map $\psi$ and $F_i$ is the above finite dimensional C*-algebra, and $k(i)$ is large enough for $\mathbb M_{k(i)}$ to contain $F_i$, then the net $\psi_i: A\to \mathbb M_{k(i)}(\mathfrak A)$ satisfies the conditions in the definition of $\mathfrak A$-quasidiagonality.   
\end{proof}

Following Tikuisis-Winter \cite{tiw}, to further explore the finiteness of module nuclear dimension and decomposition rank, we explore these notions for module maps.   

\begin{definition} \label{nuclear2}
	For $\mathfrak A$-C*-algebras $A$, $B$ a module map $\phi: A\to B$  has module nuclear dimension $n$, writing ${\rm dim}_{\rm nuc}^\mathfrak A(\phi)=n$,  if $n$ is the smallest non-negative integer such that for every finite set $\mathcal F\ssubset A$ and $\varepsilon > 0$ there
	exist a finite dimensional C*-algebra $F$, c.c.p. admissible map $\psi: A \to  F \otimes \mathfrak A$, a nonempty set $\mathcal G$ of c.p. admissible maps $\varphi: F\otimes \mathfrak A\to B$ such that $\|\varphi\circ\psi(a) - \phi(a)\|< \varepsilon$, for $a\in \mathcal F$, $\varphi\in\mathcal G$, and $F$ decomposes into $n+1$ ideals $F = F^{(0)}\oplus \cdots\oplus F^{(n)}$  such that for each $i = 0, \cdots, n$,
	any compact set $K_i\subseteq (F^{(i)}\otimes \mathcal A)_1$, and any $\varepsilon_i > 0$, there is a $\varphi\in\mathcal G$ whose restrictions $\varphi^{(i)}: F^{(i)}\otimes \mathfrak A\to B$ are $(\varepsilon_i, K_i)$-order 0.
	If moreover the family $\mathcal G$ could be arranged to consist of c.p.c module maps $\varphi: F\otimes \mathfrak A\to B$, we say that module decomposition rank of $A$ (over $\mathfrak A$) is $n$, writing ${\rm dr}_\mathfrak A(\phi)=n$.  
\end{definition}

Let $\iota_A: A\hookrightarrow A_\infty$ be the embedding into the cosets of constant sequences. 
The following lemma is a restatement of Corollary \ref{rel}.

\begin{lemma} \label{equal}
${\rm dim}_{\rm nuc}^\mathfrak A(A)={\rm dim}_{\rm nuc}^\mathfrak A(\iota_A)$.
\end{lemma}

For $\mathfrak A$-C*-algebras  $A, B$, two admissible $*$-homomorphisms $\sigma, \pi: A\to B$ are $\mathfrak A$-homotopic if there is a continuous path of admissible $*$-homomorphisms $\sigma_t: A\to B$ with  $\sigma_0=\sigma$ and $\sigma_1=\pi$. We say that $A$ and $B$ are $\mathfrak A$-homotopic, if there are  admissible $*$-homomorphisms $\sigma: A\to B$ and  $\pi: B\to A$, with $\pi\circ\sigma$ and $\sigma\circ\pi$ being $\mathfrak A$-homotopic to ${\rm id}_A$ and ${\rm id}_B$, respectively. 

\begin{lemma}
Let $A$ and $B$ be separable. If admissible $*$-homomorphisms $\sigma_0, \sigma_1: A\to B$ are $\mathfrak A$-homotopic, $\sigma_0$ is injective, and $\sigma_1(A)$ is $\mathfrak A$-QD, then $A$ is $\mathfrak A$-QD. 
\end{lemma}	
\begin{proof}
We just sketch the proof to highlight where the module structure has to be considered. By \cite[Lemma 6.4]{lan}, we could find a faithful essential representation $B\subseteq \mathbb B(K\otimes \mathfrak A)$, for some separable Hilbert space $K$. Given finite subset $\mathcal F\ssubset A$ and $\varepsilon>0$, there is a separable Hilbert space $H$, an admissible $*$-homomorphisms $\pi: A\to \mathbb B(H\otimes \mathfrak A)$, and a finite rank projection $p\in \mathbb B(H\otimes \mathfrak A)$ with,
$$\|[p, \pi(a)]\|<\varepsilon, \ \ \|p\pi(a)p\|>\|a\|-\varepsilon,\ \ (a\in\mathcal F).$$
Given a continuous path $(\sigma_t)$ of admissible $*$-homomorphisms from $\sigma_0$ to $\sigma_1$, and $\delta>0$, by injectivity of $\sigma_0$, 
$$\|a\|=\|\sigma_0(a)\|=\sup_{q}\|q\sigma_0(a)q\|,$$
where the supremum runs over all finite rank projections $q\in \mathbb B(H\otimes \mathfrak A)$. Thus there is such a projection with,
$$\|a\|-\varepsilon\leq\|q\sigma_0(a)q\|,\ \ (a\in\mathcal F).$$
Apply \cite[Lemma 7.3.1]{bo} to the separable ideal $\mathbb F(K\otimes \mathfrak A)\unlhd
\mathbb B(K\otimes \mathfrak A)$ consisting of finite rank operators to find positive norm one, finite rank operators $q_0,\cdots, q_n$ in $\mathbb B(K\otimes \mathfrak A)$ with $q\leq q_0\leq \cdots\leq q_n\leq 1,$, $q_{j+1}q_j=q_j$, $\|[q_j, \sigma_{j/n}(a)]\|<\delta$, for each $0\leq j\leq n$, $a\in\mathcal F$, and use \cite[Lemma 3.4]{a3} to $\sigma_1(A)\subseteq \mathbb B(K\otimes \mathfrak A)$ observe that 
$q_n$ could be chosen to be a projection. Put $$v:=q_0^{\frac{1}{2}}\oplus (q_1-q_0)^{\frac{1}{2}}\oplus\cdots\oplus(q_n-q_{n-1})^{\frac{1}{2}},$$
and put $p:=vv^*$ and $\pi:=\sigma_0\oplus\sigma_{\frac{1}{n}}\oplus\sigma_{\frac{2}{n}}\cdots\sigma_1$. The rest of the proof goes verbatim to that of \cite[Proposition 7.3.5]{bo}.
\end{proof}

\begin{corollary} \label{aqd}
If $A$ is separable, the cone $CA$ is always $\mathfrak A$-QD.
\end{corollary}

\begin{remark}\label{mod}
	$(i)$ In the above lemma, if the  net of approximately multiplicative and approximately isometric c.c.p. maps $\phi_n: \sigma_1(A) \to \mathbb M_{k(n)}(\mathfrak A)$  could be arranged to be module maps, same holds for $A$: take a bounded approximate identity $(p_i)$ of $\mathbb K(K)$ consisting of finite rank projections and put $h:=\sum_i 2^{-i}(p_i\otimes 1_\mathfrak A)$. Then $h\in \mathbb F(K\otimes \mathfrak A)$ is a strictly positive element, and,
	$$h(\xi\otimes\alpha\beta)=\sum_i 2^{-i}(p_i\xi\otimes \alpha\beta)=\sum_i 2^{-i}(p_i\xi\otimes \alpha)\cdot\beta=h(\xi\otimes\alpha)\cdot\beta,$$ 
	i.e., $h$ is a module map. For $n\geq 1$, let $f_n\in C_0(0,1]$ be zero on $(0,2^{-n}]$, one on $[2^{n-1},1]$, and linear elsewhere, then for $e_n:=f_n(h),$ since $f_n(0)=0$, $e_n$ is a limit of polynomials in $h$ without constant term, thus $\mathbb F(K\otimes \mathfrak A)$ has a quasi-central approximate identity consisting of module maps. In particular, the finite rank operators $q_0,\cdots,q_{n-1}$ could be chosen to be module maps. Now, since the approximating sequence $(\phi_n)$ for $\sigma_1(A)$ is also assumed to be a module map, the corresponding projection $q_n$ could also be arranged to be a module map on $K\otimes \mathfrak A$. Next, $p=vv^*$ could be represented as an $n\times n$ matrix with components being 0 or of the form $(q_i-q_{i-1})^{\frac{1}{2}}(q_j-q_{j-1})^{\frac{1}{2}}$, for $0\leq i,j\leq n$, with $q_{-1}:=0$ (see the proof of \cite[Proposition 7.3.5]{bo}). Therefore, $p$ is also a module map on $H\otimes\mathfrak A$, as required.
	
	$(ii)$ Since $CA$ is homotopic to zero, by $(i)$ above, we may always choose the  approximating sequence $(\phi_n)$ for $CA$ to be module maps. 	       
\end{remark}

\begin{lemma} \label{pi}
Let $A$ be a unital, simple, $\mathfrak A$-nuclear, and purely infinite. Then there exists a
sequence c.c.p. module maps $\phi_n : A\to A$ which factorize through $\mathbb M_{k(n)}(\mathfrak A)$ as $\psi_n=\eta_n\circ\theta_n$
with $\theta_n$ c.c.p. module maps and $\eta_n$ 
$*$-homomorphism and module map, such that  $\Psi :=
(\psi_n): A\to A_\omega$ is a c.c.p. order 0 module map.
\end{lemma}
\begin{proof}
By Remark \ref{mod}, the cone $CA$ is $\mathfrak A$-QD with approximately multiplicative and approximately isometric sequence $\phi_n: CA\to \mathbb M_{k(n)}(\mathfrak A)$ consisting of module maps. Let $$\Phi:=(\phi_n): CA\to \prod_{n\in\omega} \mathbb M_{k(n)}(\mathfrak A)$$ be the corresponding isometric c.c.p. order 0 module map. This composed with the isometric $*$-homomorphism and module map $\iota: a\mapsto {\rm id}\otimes a$ gives an isometric c.c.p. order 0 module map $\Phi\circ \iota$ on $A$. Write $\Phi\circ \iota=:(\theta_n)$, where $\theta_n: A\to \mathbb M_{k(n)}(\mathfrak A)$ is c.c.p. order zero module map.   
 	
 Next, since $A$ is purely infinite and unital, it contains a copy of the Cuntz algebra $\mathcal O_\infty$, thus these is a $*$-monomorphism $\zeta_n: \mathbb M_{k(n)}(\mathbb C)\hookrightarrow A$. Consider $\mathfrak A$ as a {\it left} module over itself, and construct the $*$-homomorphism and module map $$q\circ(\zeta_n\otimes {\rm id}_\mathfrak A): \mathbb M_{k(n)}(\mathfrak A)\to A\otimes \mathfrak A\twoheadrightarrow	 A\otimes_\mathfrak A \mathfrak A,$$ where $q$ is the quotient map, and let us observe that the range is $*$-isomorphic to $A$. Indeed, given $x:=\Sigma_i a_i\otimes \alpha\in A\odot_\mathfrak A \mathfrak A,$ for $y:=\Sigma_i a_i\cdot \alpha\in A$, we have $x=y\otimes  1_{\mathfrak A}$, thus the map $\Sigma_i a_i\otimes \alpha\mapsto \Sigma_i a_i\cdot \alpha$ is well-defined $*$-homomorphism, module map, and isometry: $A\odot_\mathfrak A \mathfrak A\to A$, which then extends by universal property \cite[Proposition 3.4]{a2} to a  $*$-isomomorphism and module map $\pi: A\otimes_\mathfrak A^{max}\mathfrak A\to A$. Finally, by \cite[Corollary 3.12]{a2}, $A\otimes_\mathfrak A^{max}\mathfrak A$ is isomorphic to $A\otimes_\mathfrak A^{min}\mathfrak A=:A\otimes_\mathfrak A\mathfrak A$. Put $\eta_n:=\pi\circ q\circ\zeta_n$ and $\psi_n:=\eta_n\circ \theta_n$.        
\end{proof}

\section{Vector valued traces and module strict comparison} \label{3}

In this section, we define strict comparison for $\mathfrak A$-C*-algebras. First let us recall the notion of module retracts from \cite{a3} (c.f., \cite[Page 55]{lan}). 

\begin{definition} \label{ret}
	An $\mathfrak A$-retraction is a positive (and so, automatically c.p.) right $\mathfrak A$-module map $\tau: A\to \mathfrak A$ such that
	
	$(i)$ Im$(\tau)$ is strictly dense in $\mathfrak A$, that is, for each $\beta\in\mathfrak A$, there is a net $(a_i)\subseteq A$ such that
	$$\|\tau(a_i\cdot\alpha)-\beta\alpha\|\to 0\ \ (\alpha\in\mathfrak A),$$
	
	$(ii)$ for some bounded approximate identity $(e_i)$ of $A$, $\tau(e_i)\to p$, for some projection $p$ in $\mathfrak A$, in the strict topology.
	
	When $\mathfrak A$ is a von Neumann algebra, a weak $\mathfrak A$-retraction is a positive right $\mathfrak A$-module map $\tau: A\to \mathfrak A$ such that Im$(\tau)$ is ultra-strongly dense in $\mathfrak A$, and for some bounded approximate identity $(e_i)$ of $A$, $\tau(e_i)\to p$, for some projection $p$, ultra-strongly in $\mathfrak A$.
	We denote the set of all (resp., weak) $\mathfrak A$-retractions by $\mathcal R(A, \mathfrak A)$ (resp., $\mathcal R^w(A, \mathfrak A)$).
\end{definition}

There is a one-one correspondence between $\mathfrak A$-retractions and c.c.p. idempotents from the multiplier algebra $M(A)$ to $\mathfrak A$, strictly continuous on the unit ball. Note that an $\mathfrak A$-retraction is automatically a bimodule map if the left module action is defined by,
$$\alpha\cdot a:=(a^*\cdot\alpha^*)^*,\ \ (\alpha\in \mathfrak A, a\in A).$$

An $\mathfrak A$-{\it trace} is an  $\mathfrak A$-retraction $\tau$ satisfying, $\tau(ab)=\tau(ba)$, for $a,b\in A$. When $A$ and $\mathfrak A$ are unital, we further assume that $\tau(1_A)=1_\mathfrak A$ (in this case, conditions $(i)$, $(ii)$ above are automatic). We say that $\tau$ is {\it faithful} if $\tau(a^*a)=0$ implies $a=0$. 

The set of all $\mathfrak A$-traces is denoted by $\mathcal T_\mathfrak A (A)$. If $\mathfrak Z:=Z(\mathfrak A)$ is the center of $\mathfrak A$, we use $\mathcal T_\mathfrak Z (A)$ (resp., $\mathcal T^w_\mathfrak Z (A)$) to denote those maps $\tau: A\to \mathfrak A$, called $\mathfrak Z$-traces, which satisfy all the conditions of a $\mathfrak A$-trace, except that they are only supposed to be a $\mathfrak Z$-module map. Similarly (weak) retractions which are only a $\mathfrak Z$-module map are denoted by $\mathcal R_\mathfrak Z(A, \mathfrak A)$ (resp., $\mathcal R^w_\mathfrak Z(A, \mathfrak A)$). When $A$ is unital and $\mathfrak A$ is a finite von Neumann algebra, there is a center valued trace $tr: \mathfrak A\to \mathfrak Z$ which is also a $\mathfrak Z$-module map \cite[Theorm 2.6]{t}, hence the convex set of $\mathfrak Z$-traces into $\mathfrak Z$   is a ``factor'' of the convex set of  $\mathfrak Z$-traces into $\mathfrak A$ (i.e., there is a surjective affine map from the second onto the first). 

Recall that for $x,y\in A_+$, $x$ is Cuntz-subequivalent to $y$, writing $x	\lesssim y$, if there is a sequence $(u_n)\subseteq A$ such that $\|u_nxu_n^*-y\|\to 0$ \cite{c}. Also $x$ is Cuntz-Pedersen subequivalent to $y$, writing $x	\lessapprox y$, if there is a sequence $(u_n)\subseteq A$ such that $x=\sum_n u_nu_n^*$, $z=\sum_n u_n^*u_n$, for some $z\leq y$ \cite{cp}. When $x	\lesssim y$ and $y	\lesssim x$, we write $x\sim y$ and say that $x$ and $y$ are Cuntz-equivalent. The Cuntz-Perdersen equivalence $x\approx y$ is defined similarly. 

An element $x\in A_+$ is said to be finite if $y\approx x$ and $y\leq x$ implies $y=x$, for each $y\in A_+$. $A$ is called {\it finite} ({\it semifinite}) in the sense of Cuntz-Pedersen if every (nonzero) positive element in $A$ is finite (dominates a nonzero finite element).

Following \cite{cp}, we write $A_0:=\{x-y: x,y\in A_+, x\approx y\}$ and $A^q_{sa}:= A_{sa}/A_0$, and put $A^q=A^q_{sa}+iA^q_{sa}$. Denote the quotient map by $q: A_{sa}\twoheadrightarrow A^q_{sa};\ a\mapsto \dot a$. This extends to a surjection  $q: A\twoheadrightarrow A^q$, and we again write $q(a)=\dot a$. 
There is a one-one correspondence (indeed an affine isomorphism of convex sets) between $\mathcal R(A^q, \mathfrak A)$  and $\mathcal T_\mathfrak A (A)$, sending  $\mathfrak A$-retraction $\phi$ on $A^q$ to  $\mathfrak A$-trace $\tau$ on $A$, defined by $\tau(a)=\phi(\dot a)$.

Let $A$ be an $\mathfrak A$-C*-bimodule such that the compatibility conditions of Section \ref{1} holds for actions from both sides (this holds if $A$ is a right $\mathfrak A$-C*-module and the left action is induced from the right action by taking adjoints). Let us recall (and define a weaker version of) the notion of central action from \cite{a2}. Let us put $\mathfrak Z=Z(\mathfrak A)$, as above. We say that $A$ is a (resp., {\it weakly}) {\it central} $\mathfrak A$-C*-bimodule if $Z(A)$ is an $\mathfrak A$-submodule (resp., $\mathfrak Z$-submodule) of $A$. When $A$ is unital, this is equivalent to requiring that $\alpha\cdot 1_A, 1_A\cdot\alpha\in Z(A)$, for each $\alpha\in \mathfrak A$ (resp., each $\alpha\in \mathfrak Z$). We also say that $A$ is (resp., {\it weakly}) {\it symmetric} if $\alpha\cdot a=a\cdot\alpha$, for each $a\in A$ and each $\alpha\in \mathfrak A$ (resp., each $\alpha\in \mathfrak Z$). When $A$ is unital and (weakly) central, this is equivalent to requiring that $\alpha\cdot 1_A= 1_A\cdot\alpha$, for each $\alpha\in \mathfrak A$ (resp., each $\alpha\in \mathfrak Z$). 

Assume that the left action is induced from the right action by taking adjoints. If $A$ is (resp., weakly) symmetric and (resp., weakly) central,  then $\mathfrak A_+\cdot A_+\subseteq A_+$ (resp., $\mathfrak Z_+\cdot A_+\subseteq A_+$), and the same for the right action: given $\alpha=\beta\beta^*$ in $\mathfrak A$ (resp., in $\mathfrak Z$) and $a=b^*b\in A$, we have,
$$a\cdot\alpha=\alpha\cdot a=(\beta\beta^*)\cdot b^*b=\beta\cdot(\beta^*\cdot b^*)b =(\beta^*\cdot b^*)b\cdot\beta =(b\cdot\beta)^*(b\cdot\beta)\geq 0,      
$$
for each $a\in A$ and each $\alpha\in \mathfrak A$ (resp., each $\alpha\in \mathfrak Z$).

We need a version of Hahn decomposition for vector valued real linear maps. To prove this, a positive Cohen factorization theorem for Banach lattices is needed. First let us recall the following result of de Jeu and Jiang \cite[Theorem 5.3]{jj}.

\begin{lemma} \label{jj}
 Let $A$ be an ordered Banach algebra with a closed positive cone and a
positive left approximate identity $(e_i)$ of bound $M \geq 1$, let $X$ be an ordered Banach
space with a closed positive cone, and $\pi: A\to \mathbb B(X)$ be a positive continuous representation, that is, $\pi(A_+)X_+\subseteq X_+$. Let $S$ be a non-empty subset of $X_+$ such that $\pi(e_i)s\to s$, uniformly
on $S$, as $i\to\infty$, and suppose that $S$ is bounded or that $(e_i)_{i\in I}$ is commutative, that is, $e_ie_j=e_je_i$, for each $i,j\in I$.
Suppose that there exists a unital ordered Banach superalgebra $B\supseteq A$ such that,

\vspace{.2cm}
($a$) the restricted positive continuous representation $\pi_e$ of $A$ on the span closure $X_e$ of $\pi(A)X$ extends to a
positive continuous unital representation of $B$ on $X_e$,

\vspace{.1cm}
($b$) $B_+$ is inverse closed in $B$, that is, $b^{-1}\in B_+$, for each invertible element $b\in B_+$.

\vspace{.2cm}
Then for every $\varepsilon > 0$, $\delta> 0$, and $0 < r < (M + 1)^{-1}$, every
integer $n_0 \geq 1$, and every sequence $(t_n)\subseteq (1, \infty)$ converging to $\infty$,
there exist $a\in A$ and a sequence of maps $x_n: S \to X$, such that,

\vspace{.2cm}
(1) $s = \pi(a^n)x_n(s)$, for $n\geq 1$ and $s\in S$,

\vspace{.1cm}
(2) $\|a\|\leq M$, 

\vspace{.1cm}
(3) each $x_n$ is a uniformly continuous homeomorphism of $S$ onto $x_n(S)$, with
the restricted map $\pi(a^n) : x_n(S) \to S$ as its inverse,

\vspace{.1cm}
(4) $\|s - x_n(s)\|\leq\varepsilon$, for $n\leq n_0$ and $s\in S$, and $\|x_n(s)\|\leq t_n^n{\rm max}(\|s\|,\delta)$ for  $n \geq 1$ and $s\in S$,

\vspace{.1cm}
(5) $x_n(s_1 + s_2) = x_n(s_1) + x_n(s_2)$,
for $n\geq 1$, whenever $s_1, s_2, s_1 + s_2\in S$, and $x_n(\lambda s) = \lambda x_n(s)$, for $n\geq 1$ and $\lambda\in \mathbb C$, whenever $s, \lambda s\in S$,

\vspace{.1cm}
(6) there exists a sequence $(u_i)\subseteq \{e_i: i\in I\}$ such that $a= \sum_{i=1}^{\infty} r(1-r)^{i-1}u_i$ is an element of the closed convex hull of $\{u_i\}$ in $A$, and for every $k \geq  0$, the element $b_k:= (1 - r)^k 1_B + \sum_{i=1}^{k} r(1-r)^{i-1}u_i$ of $B$ is
in the convex hull of $\{1_B, u_1, \cdots, u_k\}$ in $B$, which is invertible in $B$,
and its inverse lies in the unital Banach subalgebra of $B$ generated by
$\{u_1, \cdots, u_k\}$, and $\pi(b_k^{-n}s)\to x_n(s)$, as $k\to\infty$, for $n\geq 1$ and $s\in S$. Moreover, there exists a monotone (strictly) increasing countable chain $(i_k)\subseteq I$ such that $u_k = e_{i_k}$, for $k\geq 1$ (whenever $I$ has no lasgest element),

\vspace{.1cm}
(7) if $S$ is (totally) bounded, so is $x_n(S)$, for each $n\geq  1$,  and $\pi(e_i)x_n(s)\to x_n(s)$,  uniformly on $S$, for $n\geq 1$, 

\vspace{.1cm}
(8) $a\in A_+$,

\vspace{.1cm}
(9) $x_n(S)\subseteq X_{e+}$, for $n\geq 1$.  
\end{lemma}

\begin{lemma} [Cohen positive factorization] \label{cf}
	If  $ A_+\cdot\mathfrak A_{+}\subseteq A_+$, then for each $x\in A_+$, there is $\alpha\in \mathfrak A_+$ and $y\in A_+$ such that $x=y\cdot\alpha$. 
\end{lemma} 
\begin{proof}
In the above lemma, let $A=\mathfrak A$, $X=A$, and $\pi(\alpha)a=a\cdot\alpha $, for $\alpha\in\mathfrak A$ and $a\in A$. The second assumption simply means that $\pi$ is a positive representation, and the fact that $A$ is a right Banach $\mathfrak A$-module means that $\pi$ is continuous. Given $x\in A_+$, put $S:=\{x\}$, and observe that by the Cohen factorization theorem (c.f., \cite[Corollary 16.2]{dw}), there is $\beta\in \mathfrak A$ and $z\in A$ with $x=y\cdot\beta$. Take a positive approximate identity $(e_i)$ in $\mathfrak A$, then $x\cdot e_i=(y\cdot\beta)\cdot e_i=y\cdot(\beta e_i)\to y\cdot\beta=x$, as $i\to\infty$. Let $B=\mathfrak A\oplus \mathbb C$ be the minimal unitization of $\mathfrak A$, then right action of $\mathfrak A$ extends to a right action of $B$, by putting $a\cdot 1_B=a$, for $a\in A$, then assumption $(a)$ of the above lemma holds. Also, since $B$ is closed under continuous functional calculus, assumption $(b)$ also holds. By parts (1), (8), and (9) of Lemma \ref{jj}, there is $\alpha\in \mathfrak A_+$ and $y:=x_1(x)\in A_+$ with $x=y\cdot\alpha$, as required.         
\end{proof}	

Recall that a C*-algebra is an AW*-algebra if  the left (resp., right) annihilator of any subset of $A$
is generated as a left (resp., right) ideal by some projection in $A$ \cite{ka}. These are characterized as those C*-algebras whose masa's are monotone complete \cite[3.9.2]{p}.
Also, recall that injective C*-algebras are monotone complete \cite[Theorem 7.1]{to}.

\begin{lemma} [Hahn decomposition] \label{mc}
	Let $\mathfrak A$ be a  monotone complete C*-algebra, then every real linear  $\psi: A_{sa}\to \mathfrak A_{sa}$ decomposes into positive parts $\psi=\psi_+-\psi_-$ with $\psi_\pm$ module maps satisfying $\|\psi\|=\|\psi_+\|+\|\psi_-\|$, and $\psi_\pm(A_+)\subseteq\mathfrak A_+$. If $\psi$ is tracial, we may choose $\psi_\pm$ to be tracial. When $A$ is weakly symmetric and  weakly central and  the left action is induced from the right one by taking adjoints, if $\psi$ has a complex linear   $\mathfrak Z$-module map extension on $A$, we may choose $\psi_\pm$ to have complex linear   $\mathfrak Z$-module map extensions on $A$. 
\end{lemma} 
\begin{proof}
Let $\psi_1, \psi_2: A_{sa}\to \mathfrak A$ be real linear (tracial) maps and define, 
$$\psi_1\vee\psi_2(x):=\sup\{\psi_1(y)+\psi_2(z): y,z\in A_+, y+z\approx x\},$$
and,
$$\psi_1\wedge\psi_2(x):=\inf\{\psi_1(y)+\psi_2(z): y,z\in A_+, y+z\approx x\},$$
for $x\in A_+$, where the sup and inf exist as $\mathfrak A$ is monotone complete. The fact that these maps have real linear (and tracial) extensions on $A_{sa}$ follows by an argument similar to \cite[proposition 2.8]{cp} (see also the proof of \cite[Theorem 3.11]{pe}).

If $B\subseteq\mathfrak A_+$ is bounded above and $\alpha\in\mathfrak Z_+$, then $\alpha B\subseteq\mathfrak A_+$ is bounded above and $\sup(\alpha B)=\alpha\sup B$: let $b_0:=\sup B$, then $\alpha b=\alpha^{\frac{1}{2}}b\alpha^{\frac{1}{2}}\leq \alpha^{\frac{1}{2}}b\alpha^{\frac{1}{2}}=\alpha b_0,$ for $b\in B$, thus, $\sup(\alpha B)\leq \alpha b_0$. We may assume that $b_0\neq 0$, and so, given $0<\varepsilon<1$, there is $b\in B$ with $b\geq (1-\varepsilon)b_0$, thus with the same argument, $\alpha b\geq (1-\varepsilon)\alpha b_0$, hence, $\sup(\alpha B)\geq (1-\varepsilon)\alpha b_0$. If $\alpha b_0\neq\sup(\alpha B)$,  $f(\alpha b_0-\sup(\alpha B))>0$, for some $f\in \mathcal S(\mathfrak A)$, which is a contradiction, since $f(\alpha b_0-\sup(\alpha B))<\varepsilon f(\alpha b_0)$, for each $0<\varepsilon<1$.

Next, let us observe that these  are $\mathfrak Z$-module maps, if $\psi_1$, $\psi_2$ are so. For  $\alpha\in\mathfrak Z_+$ and $x\in A_+$, by the above observation, 
\begin{align*}
\psi_1\vee\psi_2(\alpha\cdot x)&=\sup\{\psi_1(y)+\psi_2(z): y,z\in A_+, y+z\approx \alpha\cdot x\}\\&\geq \sup\{\psi_1(\alpha\cdot y^{'})+\psi_2(\alpha\cdot z^{'}): y^{'},z^{'}\in A_+, y^{'}+z^{'}\approx x\}\\&\geq \alpha\sup\{\psi_1(y^{'})+\psi_2(z^{'}): y^{'},z^{'}\in A_+, y^{'}+z^{'}\approx x\}\\&=\alpha(\psi_1\vee\psi_2(x)),
\end{align*} 
where the first inequality follows from the fact that for $y^{'},z^{'}\in A_+, y^{'}+z^{'}\approx x$, there is a sequence $(u_n)\subseteq A_+$ with $y^{'}+z^{'}=\Sigma_n u_nu_n^*$ and $x= \Sigma_n u_n^*u_n$, and since $A$ is  weakly symmetric and  weakly central, 
$$y:=\alpha\cdot y^{'}=(\alpha^{\frac{1}{2}}\alpha^{\frac{1}{2}})\cdot(y^{'\frac{1}{2}}y^{'\frac{1}{2}})=(\alpha^{\frac{1}{2}}\cdot y^{'\frac{1}{2}})^2=
 (\alpha^{\frac{1}{2}}\cdot y^{'\frac{1}{2}})^*(\alpha^{\frac{1}{2}}\cdot y^{'\frac{1}{2}})\geq 0,$$   
 similarly $z\geq 0$, and 
 $$y+z=\alpha\cdot (y^{'}+z^{'})=\alpha\cdot\Sigma_n u_nu_n^*=\Sigma_n (\alpha^{\frac{1}{2}}\cdot u_n)(\alpha^{\frac{1}{2}}\cdot u_n)^*,$$
and similarly,
$\alpha\cdot x=\Sigma_n (\alpha^{\frac{1}{2}}\cdot u_n)^*(\alpha^{\frac{1}{2}}\cdot u_n),$
that is, $y+z\approx \alpha\cdot x$. 

Conversely, without loss of generality we may assume that $\mathfrak A$ is unital, and if $y,z\in A_+$, $\alpha\in\mathfrak Z_+$, and $y+z\approx \alpha\cdot x$, then for $\alpha_n:=(\alpha+\frac{1}{n}1_\mathfrak A)^{-1}\in \mathfrak Z_+$, by the same argument as above, $\alpha_n\cdot y+\alpha_n\cdot z\approx\alpha\alpha_n\cdot x$. 

Let $\mathfrak Y$ be the C*-subalgebra of $\mathfrak Z$ generated by the positive element $\alpha$. Consider $A$ as a $\mathfrak Y$-module and apply Lemma \ref{cf} (with $\mathfrak A$ replaced by $\mathfrak Y$ and the right action replaced by the left action) to find $\beta\in \mathfrak Y_+$ and $y^{'}\in A_+$ with $y=\beta\cdot y^{'}$. Since elements of the form $P(\alpha)$, where $P$ is a complex polynomial with $P(0)=0$ are dense in $\mathfrak Y$, and since for a such a polynomial $P$, there is a polynomial $Q$ with $P(z)=zQ(z)$, there is $\gamma\in\mathfrak Y$ with $\beta=\alpha\gamma$, and we clearly could choose $\gamma$ to be positive. Thus $y=(\alpha\gamma)\cdot y^{'}$, and if we replace $y^{'}$ in the above decomposition by $\gamma\cdot y^{'}$ (and call this element $y^{'}$ again), we have $y=\alpha\cdot y^{'}$, for some $y^{'}\in A_+$. Similarly, $z=\alpha\cdot z^{'}$, for some $z^{'}\in A_+$. 

Let us rewrite $\alpha_n\cdot y+\alpha_n\cdot z\approx\alpha\alpha_n\cdot x$ as $\alpha\alpha_n\cdot y^{'}+\alpha\alpha_n\cdot z^{'}\approx\alpha\alpha_n\cdot x$. Consider the universal representation $\mathfrak Z\subseteq \mathbb B(H_u)$, for the universal Hilbert space $H_u$ (i.e., the direct sum of Hilbert spaces of GNS-representations associated to states on $\mathfrak A$), and note that,
$$\alpha\alpha_n=\alpha(\alpha+1/n1_\mathfrak A)^{-1}=(\alpha+1/n1_\mathfrak A)^{-1/2}\alpha(\alpha+1/n1_\mathfrak A)^{-1/2}\to p,$$ 
strongly in $\mathbb B(H)$, where $p\in \mathbb B(H)$ is the support projection of $\alpha$, that is the smallest projection satisfying $\alpha p=p\alpha=\alpha$ \cite[5.4.7]{s}. 

Let us observe that $p\in\mathfrak Z$. First note that the center of a monotone complete C*-algebra is again monotone complete: Given $(\beta_n)\in\mathfrak Z$, let $\beta:=\sup(\beta_n)\in\mathfrak A$, then for each unitary $u\in\mathfrak A$, taking sup from both sides of  $\beta_n=u\beta_nu^*$ in $\mathfrak A$, we get $\beta=u\beta u^*$, i.e.,  $\beta$ commutes with each unitary in $\mathfrak A$, thus, $\beta\in Z(\mathfrak A)=\mathfrak Z$. Now let $q:=\sup(\alpha\alpha_n)\in\mathfrak Z_+$, then since $p$ commutes with $\alpha$ (and so with each $\alpha_n$), 
$$\alpha \alpha_n= p\alpha \alpha_n = p^{\frac{1}{2}}\alpha\alpha_n p^{\frac{1}{2}},$$
and taking supremum in $\mathfrak Z^{**}$, we get 
$q=p^{\frac{1}{2}}q p^{\frac{1}{2}}$. Also, for each $\xi\in H$,
$$\langle q\xi,\xi\rangle-\langle (\alpha\alpha_n)\xi,\xi\rangle=\langle (q-\alpha\alpha_n)\xi,\xi\rangle\geq 0,$$
and taking limit, we get $\langle q\xi,\xi\rangle\geq\langle p\xi,\xi\rangle,$ that is, $q\geq p$. Finally, since $\alpha\alpha_n=\alpha(\alpha+1/n1_\mathfrak A)\leq 1_\mathfrak A$, we have $q\leq 1_\mathfrak A$, and so, $q=p^{\frac{1}{2}}q p^{\frac{1}{2}}\leq p$. Therefore, $p=q\in\mathfrak Z$, as claimed. 

Back to equivalence, $\alpha\alpha_n\cdot y^{'}+\alpha\alpha_n\cdot z^{'}\approx\alpha\alpha_n\cdot x$,  we get $p\cdot y^{'}+p\cdot z^{'}\approx p\cdot x$ in $A$. Repeating the  argument resulting in decomposition $y=\alpha\cdot y^{'}$ with $y$ replaced by $y^{'}$, there is a positive element $y^{''}\in A_+$ with $y^{'}=\alpha\cdot y^{''}$. Similarly, there are positive elements $z^{''}, x^{'}\in A_+$ with $z^{'}=\alpha\cdot z^{''}$ and  $x=\alpha\cdot x^{'}$. Hence, 
\begin{align*}
y^{'}+z^{'}&=\alpha\cdot y^{''}+\alpha\cdot z^{''}=p\alpha\cdot y^{''}+p\alpha\cdot z^{''}=p\cdot y^{'}+p\cdot z^{'}\\&\approx p\cdot x= p(\alpha\cdot x^{'})= p\alpha\cdot x^{'}=\alpha\cdot x^{'}=x,	
\end{align*}
in $A$. Therefore,
 \begin{align*}
 	\alpha(\psi_1\vee\psi_2(x))&=\alpha^{\frac{1}{2}}(\psi_1\vee\psi_2(x))\alpha^{\frac{1}{2}}\\&\geq\alpha^{\frac{1}{2}}(\psi_1(y^{'})+\psi_2(z^{'}))\alpha^{\frac{1}{2}}\\&=\alpha\big(\psi_1(y^{'})+\psi_2(z^{'})\big) \\&=\psi_1(\alpha\cdot y^{'})+\psi_2(\alpha\cdot z^{'})
 	\\&=\psi_1(y)+\psi_2(z),	
 \end{align*}
and taking supremum over all $y,z\in A_+$ with $y+z\approx x$, we get 
$$	\alpha(\psi_1\vee\psi_2(x))\geq\psi_1\vee\psi_2(\alpha\cdot x),$$
thus equality holds for each $\alpha\in\mathfrak Z_+$ and $x\in A_+$. Since all the maps involved are real linear, we also have the equality for each $\alpha\in\mathfrak Z_{sa}$ and $x\in A_{sa}$, that is, the complex linear extension of $\psi_1\vee\psi_2$ on $A$ is a $\mathfrak Z$-module map, if $\psi_1$ and $\psi_2$ are so. The same holds for $\psi_1\wedge\psi_2$. Finally, note that $$\psi_1\vee\psi_2+\psi_1\wedge\psi_2=\psi_1+\psi_2,$$ and put $\psi_+:=\psi\vee 0,$ and $\psi_-=(-\psi)\vee 0$. Then $\psi=\psi_+-\psi_-$ and, given $\varepsilon > 0$ there are $x, y \in A_{1,+}$ with, $$\|\psi\|\leq\|\psi_+\|+\|\psi_-\|\leq
\varepsilon+\psi(x)-\psi(y)\leq \varepsilon+\|\psi\|\|x-y\|\leq\varepsilon+\|\psi\|,$$
therefore, $\|\psi\|=\|\psi_+\|+\|\psi_-\|$, as required.
	\end{proof}

For C*-algebra $\mathfrak A$, recall that a faithful center-valued trace is a $\mathfrak Z$-module map ${\rm tr}: \mathfrak A\to \mathfrak Z$ satisfying:

(1) ${\rm tr}(\alpha\alpha^*)={\rm tr}(\alpha^*\alpha)\geq 0$,

(2) ${\rm tr}(\alpha\alpha^*)=0$ iff $\alpha=0$. 

\noindent When $\mathfrak A$ is unital, we also assume that ${\rm tr}(1_\mathfrak A)=1_\mathfrak A.$ 
  
\begin{lemma}\label{inf}
Let $A$ be unital and $\mathfrak A$ be a unital injective C*-algebra. Assume that $\mathfrak A$ admits a faithful center-valued trace. Then, for $x\in A_+$,
$$\inf_{z\in A_0}\|x-z\|=\inf_{x\lessapprox y\in A_+}\|y\|=\inf_{f\in \mathcal S(\mathfrak A)}\sup_{\tau\in\mathcal T_\mathfrak A(A)_1} f(\tau(x)),$$
where $\mathcal S(\mathfrak A)$ is the state space of $\mathfrak A$ and $\mathcal T_\mathfrak A(A)_1$ is the set of $\mathfrak A$-traces with norm at most 1.  	
\end{lemma}
\begin{proof}
	Let us denote these three values from left to right by $c_1, c_2,$ and $c_3$. Then $c_1=c_2$ \cite{cp}. If $x\lessapprox y$, then  $x=\Sigma_n u_nu_n^*$, $z=\Sigma_n u_n^*u_n$, for some $z\leq y$, thus
	$$f(\tau(x))=f(\tau(\Sigma_n u_nu_n^*))=f(\tau(\Sigma_n u_n^*u_n))=f(\tau(z))\leq f(\tau(y))\leq \|y\|,$$
	for each $f\in \mathcal S(\mathfrak A)$ and $\tau\in\mathcal T_\mathfrak A(A)_1$, that is, $c_3\leq c_2$.

	Without loss of generality, we may assume that $\dot x$ and $\dot 1$ are $\mathfrak A$-linearly independent. Consider the single generated submodule of $A^q$ generated by $\dot x$ and $\dot 1$, and define a module map $\phi$ from this submodule to $\mathfrak A$ by $\phi(\dot x\cdot\alpha+\dot 1\cdot\beta):=c_1\alpha+\beta$. Since $\mathfrak A_{sa}$ is also injective in the category of real $\mathfrak A$-modules (by a real version of \cite[Theorem 3.2]{fp}), this extends to a real-linear module map $\phi: A^q_{sa}\to \mathfrak A_{sa}$. Since $\mathfrak A$ is injective, it is also an AW*-algebra \cite[IV.2.1.7]{bl}. Now there is a maximal abelian saubalgebra (masa) $\mathfrak B$ in $\mathfrak A$ containing the center $\mathfrak Z$.  By assumption, $\mathfrak B$ is monotone complete, and so is $\mathfrak Z$ (by an argument as in the proof of  Lemma \ref{mc}). Hence, Lemma \ref{mc} gives a decomposition $\phi=\phi_+-\phi_-$ with $\phi_\pm$ positive. Define $\tau: A\to \mathfrak A$ by $\tau(a)=\phi(\dot a_1)+i\phi(\dot a_2)$, for $a=a_1+ia_2\in A$. Since $aa^*\approx a^*a$, for each $a\in A$, this is a tracial linear $\mathfrak A$-module  map. Let ${\rm tr}: \mathfrak A\to \mathfrak Z$ be a faithful center-valued trace, which  is a $\mathfrak Z$-module map by definition. Then $\psi:={\rm tr}\circ\phi: A^q\to \mathfrak Z$ is a $\mathfrak Z$-module map with $\psi(A^q_{sa})\subseteq \mathfrak Z_{sa}$, and we have the decomposition $\psi=\psi_+-\psi_-$ on $A^q_{sa}$, where each $\psi_\pm:={\rm tr}\circ\phi_\pm$ has a complex linear extension to a positive linear $\mathfrak Z$-module map $\psi_\pm: A^q\to \mathfrak Z$, which is then automatically c.p. \cite[2.2.5]{li}. In particular, $\|\psi_\pm\|=\|\psi_\pm(\dot 1)\|$ \cite[2.1.6]{li}. On the other hand,
	\begin{align*}
		\psi_-(\dot 1)&=(-\psi\vee 0)(\dot 1)\\&=\sup\{-\psi(y): y,z\in A^q_+, y+z\approx \dot 1\}\\&=-\inf\{\psi(y): y,z\in A^q_+, y+z\approx \dot 1\},
	\end{align*}
	But for $0<\varepsilon<1$,  $y:=\varepsilon \dot 1$, $z:=(1-\varepsilon)\dot 1$ are positive and $y+z=x$, thus,
	$$\inf\{\psi(y): y,z\in A^q_+, y+z\approx \dot 1\}\leq \varepsilon\psi(\dot 1) =\varepsilon{\rm tr}(\phi(\dot 1))=\varepsilon{\rm tr}(1_\mathfrak A)=\varepsilon 1_\mathfrak A,$$
	for each  $0<\varepsilon<1$, and so is zero, i.e., $\|\psi_-\|=\|\psi_-(\dot 1)\|=0$, therefore, for each $y\in A^q_+$,
	${\rm tr}(\phi_-(y))=\psi_-(y)=0$, which implies that ${\rm tr}(\phi_-(y))=0$, as ${\rm tr}$ is faithful. Hence, $\phi_-=0$, i.e., $\phi=\phi_+$ on $A^q_{sa}$, thus, $\tau$ is positive. Also $\phi$ is a $\mathfrak A$-module map, and so is $\tau$. Finally, since $A$ is unital, $\tau$ is automatically an $\mathfrak A$-retraction, i.e., $\tau$ is a $\mathfrak A$-trace. It follows that $\tau$ is c.p. \cite[Lemma 2.2]{a3}, thus, $\|\tau\|=\|\tau(1)\|=\|\phi(\dot 1)\|=\|1_\mathfrak A\|=1.$ 
	
	Finally, given $\varepsilon>0$, we may choose a state  $f\in \mathcal S(\mathfrak A)$ with $f(\tau(x))\leq c_3+\varepsilon.$ Thus, $$c_3+\varepsilon\geq f(\tau(x))=f(\phi(\dot x))=f(c_11_\mathfrak A)=c_1=c_2\geq c_3,$$ 
therefore, $c_1=c_2=c_3$, as claimed.	    
\end{proof}

 \begin{remark} \label{r1}
 	$(i)$ In the above lemma, under the weaker condition that $\mathfrak Z:=Z(\mathfrak A)$ is an injective C*-algebra, we could modify the above proof by directly defining $\phi$ on the $\mathfrak Z$-submodule generated by $\dot x$ and $\dot 1$ as above and extending it to a $\mathfrak Z$-module map $\phi: A^q\to \mathfrak Z$. The rest of argument works (and we don't need the center-valued trace ${\rm tr}$), with the difference that $\tau:=\phi\circ q$ is only a $\mathfrak Z$-trace, and the last term is replaced by $\inf_{f\in \mathcal S(\mathfrak Z)}\sup_{\tau\in\mathcal T_\mathfrak Z(A)_1} f(\tau(x)).$ In particular, this holds if $\mathfrak A$ has trivial center. 
 	
 	$(ii)$ In the above lemma, if $\mathfrak A$ is a finite injective von Neumann algebra, then the result holds, as in this case, $\mathfrak A$ always admits a faithful unital center-valued trace \cite[Theorem 2.6]{t} (working with quasitraces instead, this even holds in more general case of a finite AW*-algebra \cite{ka}; see also chapter 6 in \cite{be}).   
 \end{remark}

\begin{theorem} \label{faithful}
	Let $A$ be unital and $\mathfrak A$ be a unital injective C*-algebra which admits a faithful center-valued trace. The following are equivalent:
	
	$(i)$ $A$ is finite in the sense of Cuntz-Pedersen, 
	
	$(ii)$ $A$ has a separating family
of $\mathfrak A$-traces,

\noindent When $A$ is separable and $\mathfrak A$ is an injective finite separable von Neumann algebra (i.e., with separable pre dual), these are also equivalent to: 

$(iii)$ $A$ has a faithful $\mathfrak A$-trace.
\end{theorem}
\begin{proof}
$(i)\Rightarrow(ii)$. If $A$ is finite, the quotient map $q: A\to A^q$ is faithful on $A_+$ \cite[Theorem 3.3]{cp}. 
For a non-zero element $x\in A_+$, we have $\dot x\neq \dot 0$, thus $c_3=c_1\neq 0$, in the above lemma, that is, $\tau(x)\neq 0$, for some norm one $\mathfrak A$-trace $\tau$.

$(ii)\Rightarrow(i)$. For a non-zero element $x\in A_+$, if $\tau(x)=0$, for each $\mathfrak A$-trace $\tau$, then in the above lemma, $c_1=c_3=0$, that is, $q(x)=0$, and so again by \cite[Theorem 3.3]{cp}, $A$ is finite.

$(i)\Rightarrow(iii)$. Choose countable sets $(a_i)$ in the unit ball $A^q_1$ and $(\xi_m)$ in the normal state space of $\mathfrak A$. Then the topology induced by the metric 
$$d(\phi,\psi):=\sum_{n,m} 2^{-(n+m)}\xi_m(\phi(a_n)-\psi(a_n)),$$
is the same as the point-ultra-strong  topology on the set $UCP_\mathfrak A(A^q,\mathfrak A)$ of u.c.p. $\mathfrak A$-module maps on the quotient space $A^q$ (which are automatically, $\mathfrak A$-retractions of norm at most one). The set $UCP_\mathfrak A(A^q,\mathfrak A)$ is point-ultra-strong closed subset of the set $\mathbb B(A^q, \mathfrak A)$ consisting of all bounded linear maps, and this latter set is known to be point-ultra-strong compact \cite[1.3.7]{bo}. Thus, $UCP_\mathfrak A(A^q,\mathfrak A)$ is separable as a compact metric space, and we may choose in it a countable dense subset $(\phi_n)$. For the faithful quotient map $q: A\to A^q$, let us observe that $\tau_0:=\sum_n 2^{-n} \phi_n\circ q $ is a faithful $\mathfrak A$-trace: given $x\in A_+$ with $\tau_0(x)=0$, we have $\phi_n(\dot x)=0$, for each $n$, and so, $\phi(\dot x)=0$, for each $\phi\in UCP_\mathfrak A(A^q,\mathfrak A)$. Now the complex linear extension of the map $\phi$ constructed in the proof of Lemma \ref{inf} (still denoted by $\phi$) is a unital $\mathfrak A$-contraction, and so u.c.p. by \cite[Lemma 2.2]{a3}. Therefore, 
$c_11_\mathfrak A=\phi(\dot x)=0$, and so, in the notation of this lemma, $\|q(x)\|=\|\dot x\|=c_1=0,$ thus, $x=0$, as $q$ is faithful.

$(iii)\Rightarrow(i)$. If there is a faithful $\mathfrak A$-trace $\tau$, then for an element $x\in A_+$ with $q(x)=0$, we have $x\in A_0$, that is, there are $y,z\in A_+$ with $y\approx z$ and $x=y-z$. Therefore, $\tau(x)=\tau(y-z)=0$, hence, $x=0$, i.e., the quotient map $q: A\to A^q$ is faithful on $A_+$, hence, $A$ is finite by \cite[Theorem 3.3]{cp}. 
\end{proof}

\begin{remark} 
	
	$(i)$ We warn the reader that finiteness in the sense of Cuntz-Pedersen is not the same as the finiteness in the sense of Kaplansky, where $A$ is finite if every isometry in $A$ is a unitary  \cite{kap}. For instance, the C*-algebra of compact operators is finite (indeed, stably finite) in the sense of Kaplansky, but it has no finite non-zero trace, and so not finite in the sense of Cuntz-Pedersen (though the two notions seem to coincide on simple, separable, unital C*-algebras; c.f., \cite[Remark 3.1]{cp}).
	
	$(ii)$ If $A$ is finite and $\mathfrak A$ has an injective center $\mathfrak Z$, then the same proof as above along with Remark \ref{r1}$(i)$ shows that $A$ has a faithful $\mathfrak Z$-trace. 
	
	$(iii)$ If $A$ is simple, each non-zero $\mathfrak A$-trace  is faithful, thus in the above result, $A$ is finite iff it has a  $\mathfrak A$-trace (which is non-zero by definition).
	
	$(iv)$ If $A$ is not unital, we could recover the above result by going to minimal unitization $\tilde A:=A+\mathbb C1\subseteq M(A)$: it follows from \cite[Proposition 3.7]{cp} that $\tilde A^q_{sa}=A^q_{sa}+\mathbb R \dot 1$, where $\dot 1=1+A_0$. Thus, the quotient map of the unitization is given by $\tilde q(a+\lambda 1)=q(a)+\lambda\dot 1$, which is faithful on $\tilde A_+$ iff $q$ is faithful on $A_+$, i.e., $\tilde A$ is finite iff $A$ is finite. On the other hand, we may uniquely extend a $\mathfrak A$-trace $\tau: A\to \mathfrak A$ to a $\tilde{\mathfrak A}$-trace $\tau: \tilde A\to \tilde{\mathfrak A}$ given by $\tilde\tau(a+\lambda 1_{\tilde A})=\tau(a)+\lambda 1_{\tilde{\mathfrak A}}$, and the traces $\tau$ separate the points of $A$ iff the extended traces $\tilde\tau$ separate the points of $\tilde A$. 
\end{remark}

\begin{lemma}
	If $\mathfrak A$ is a monotone complete C*-algebra, then for each bounded above, increasing sequence $(\alpha_n)\subseteq \mathfrak A_+$, both the norm limit $\lim_n \alpha_n$ and supremum $\sup(\alpha_n)$ exist and are equal. 
\end{lemma}
\begin{proof}
	Since $\mathfrak A$ is monotone complete $\alpha:=\sup(\alpha_n)$ exists. Given $0\varepsilon<1$, there is $n_0\geq 1$ with $\alpha_{n_0}\geq (1-\varepsilon)\alpha$, thus for $n\geq n_0$, $\alpha_n\geq\alpha_{n_0}\geq (1-\varepsilon)\alpha,$ that is, $0\leq \alpha-\alpha_n\leq\varepsilon\alpha$, thus $\|\alpha-\alpha_n\|\leq \varepsilon\|\alpha\|$, for $n\geq n_0$, i.e., $\lim_n \alpha_n=\alpha$.  
\end{proof}	
	
\begin{definition}
An element of $\mathfrak A$ is called strictly positive, writing $\alpha>0$, if $f(\alpha)>0$, for each state $f\in\mathcal S(\mathfrak A)$. 
\end{definition}

A C*-algebra $\mathfrak A$ has a strictly positive element iff it is $\sigma$-unital, i.e., it has a countable bounded approximate identity: indeed, if $\alpha>0$, then $(\alpha^{\frac{1}{n}})$ is a countable bai;  conversely, if $(\alpha_n)$ is a countable bai with norm bound 1, $\alpha:=\sum_n 2^{-n}\alpha_n>0$.   

Given an $\mathfrak A$-trace $\tau: A\to \mathfrak A$, we put 
$$d_\tau(a)=\lim_{n\to \infty}\tau(a^{\frac{1}{n}})\ \ (a\in A_+),$$
and call it the dimension function of $\tau$. Since the sequence $(a^{\frac{1}{n}})\subseteq A_+$ is increasing for $0\leq a \leq 1,$ bounded above by $1\in \tilde A$, the same holds for $\tau(a^{\frac{1}{n}})$, and  the limit exists by the above lemma. The general case follows by considering $a/\|a\|$ instead of $a$.     	
	
\begin{definition}
	Let $A$ be unital. An $\mathfrak A$-quasitrace on $A$ is a function $\tau: A\to\mathfrak A$ such that,
	
	(1)  $\tau(1_A)=1_\mathfrak A$ and $\tau(x^*x) = \tau(xx^*)\geq 0$, for $x\in A$,
	
	(2) $\tau$ preserves the $\mathfrak A$-action on $A$ and is linear on any commutative C*-subalgebras of A,
	
	(3) If $x = a + ib$, where $a, b\in A_{sa}$, $\tau(x) = \tau(a) + i\tau(b),$
	
	(4) $\tau(\cdot)=\tau_2(\cdot\otimes e_{11})$, for some function $\tau_2:\mathbb M_n(A)\to \mathfrak A$, satisfying (1)-(3) above. 
\end{definition}
	
	We denote the set of $\mathfrak A$-quasitraces on $A$ by $\mathcal{QT}_\mathfrak A(A)$. If $\tau$ satisfies all the above properties except that it only preserves the $\mathfrak Z$-action on $A$, we call it a $\mathfrak Z$-quasitrace and write $\tau\in \mathcal{QT}_\mathfrak Z(A)$. With a slight abuse of notation, we also use the same notation to show the set of  $\mathfrak Z$-module quasi-traces into $\mathfrak Z$. 
	
	Note that in some references, a complex valued function satisfying (1)-(3) is called a 1-quasitrace and if it satisfies also (4), then it is called a 2-quasitrace (c.f., \cite{bh}). By an argument essentially due to S. Berberian, one could show that when $A$ is an AW*-algebra,  each $\mathfrak A$-quasitrace is automatically continuous (c.f., \cite[II.1.6]{bh}). Also, on AW*-algebras, each 1-$\mathfrak A$-quasitrace   is automatically a 2-$\mathfrak A$-quasitrace (mimic the proof of \cite[II.1.6]{bh}).   
	
	\vspace{.2cm}
There are cases where an $\mathfrak A$-quasitrace is automatically linear, and so an $\mathfrak A$-trace. Also being a (quasi)trace is preserved under composition with c.p. order 0 module maps.

 \begin{lemma} \label{lin}
 	$(i)$ If $\tau\in \mathcal{QT}_\mathfrak A(A)$ and $\tau\leq \tau_0$, for some $\tau_0\in\mathcal{T}_\mathfrak A(A)$, then $\tau\in\mathcal{T}_\mathfrak A(A)$. 
 	
 	$(ii)$ If $\phi: A\to B$ is a c.p. order zero module map, then for each $\mathfrak A$-(quasi)trace $\tau$ on $B$, $\tau\circ\phi$ is an $\mathfrak A$-(quasi)trace $\tau$ on $A$.
 \end{lemma}	
 \begin{proof}
 $(i)$ We have $f\circ\tau\leq f\circ\tau_0$, for each $f\in\mathcal S(\mathfrak A)$. It follows from \cite[II.1.9, II.2.4]{bh} that $f\circ\tau$ is linear. Now it follows that $\tau$ is also linear, since the states separate the points of $\mathfrak A$.
 
 $(ii)$ If $\tau$  is $\mathfrak A$-trace $\tau$ on $B$, for $a, b\in A_+$, by \cite[3.2]{wz2},
 $$\tau(\phi(ab))=\tau(\phi(a)^{\frac{1}{2}}\phi(b)^{\frac{1}{2}})=\tau(\phi(b)^{\frac{1}{2}}\phi(a)^{\frac{1}{2}})=\tau(\phi(ba)),$$
 i.e., $\tau\circ\phi$ is tracial. The other properties transfer immediately. 
 
 If $\tau$ is only a quasitrace, since $\phi$ extends to a c.p.c. order 0
 map $\phi^{(2)}: \mathbb M_2(A)\to \mathbb M_2(B)$ by \cite[3.3]{wz2}, $\tau\circ\phi$ extends to $\mathbb M_2(A)$. Also, $\tau\circ\phi$ is additive on commuting
 elements: if $a, b\in A$ commute, so does $\phi(a)$ and $\phi(b)$, by \cite[2.3]{wz2}, thus, 
$$\tau\circ\phi(a + b) = \tau\big(\phi(a) + \phi(b))\big) = \tau\circ\phi(a) + \tau\circ\phi(b),$$
as required.  
 \end{proof}	

Here is another case where this happens.
	
\begin{lemma} [Haagerup] \label{qt}
	If $A$ is unital and exact, $\mathcal{QT}_\mathfrak A(A)=\mathcal{T}_\mathfrak A(A)$. In particular, if $\mathfrak A$ is unital and exact, then $\mathcal{QT}_\mathfrak A(\mathbb M_n(\mathfrak A))=\mathcal{T}_\mathfrak A(\mathbb M_n(\mathfrak A))$.
\end{lemma}	
\begin{proof}
Let $\tau\in \mathcal{QT}_\mathfrak A(A)$. Given $f\in \mathcal S(\mathfrak A)$, $f\circ \tau$ is a normalized quasitrace on $A$ in the sense of Haagerup \cite[Definition 3.1]{h} (that is, a 2-quasitrace in the sense of Blackadar-Handelman \cite{bh}). It follows from \cite[Theorem 5.11]{h} that $f\circ \tau$ is linear on $A$, and so is $\tau$ (since states separates the points of $\mathfrak A$). 
\end{proof}

The next lemma is proved exactly as \cite[II.1.11]{bh}. In the above cited reference, this is proved when $A$ is an AW*-algebra, but the same proof also works in general (c.f., \cite[II.2.5]{bh}). Here we reproduce the proof for the sake of completeness. 
 
\begin{lemma}\label{sum}
If $\tau: A\to \mathfrak A$ is an $\mathfrak A$-quasitrace, 
$$\tau(a+b)\leq 2(\tau(a)+\tau(b))\ \ (a,b\in A_+).$$
\end{lemma}
\begin{proof}
For $a,b\in A_+$ and matrix units $e_{ij}\in \mathbb M_2(A)$, $1\leq i,j\leq 2$, let $$x:=a^{\frac{1}{2}}e_{11}+b^{\frac{1}{2}}e_{21}, \ y:=2ae_{11}+2be_{22}, \ z:=a^{\frac{1}{2}}e_{11}-b^{\frac{1}{2}}e_{21},$$ then $x^*x=(a+b)e_{11}$, whereas, $$xx^*=ae_{11}+a^{\frac{1}{2}}b^{\frac{1}{2}}e_{12}+b^{\frac{1}{2}}a^{\frac{1}{2}}e_{21}+be_{22},$$
and $xx^*+zz^*=y$. Thus,
$$\tau(a+b)=\tau(x^*x)=\tau(xx^*)\leq \tau(y)=2(\tau(a)+\tau(b)),$$ as claimed. 
\end{proof}

\begin{theorem}
Let $\mathfrak A$ be unital, exact, and simple and $A$ be separable with $d:={\rm dr}_\mathfrak A A<\infty$, then $\mathcal{QT}_\mathfrak A(A)=\mathcal{T}_\mathfrak A(A)$.
\end{theorem}
\begin{proof}
By Lemma \ref{key}, for each $n\geq 1$ and $\varepsilon > 0$, there is a finite dimensional C*-algebra $F_n = F_n^{(0)}\oplus \cdots\oplus F_n^{(d)}$, with each $F_n^{(i)}$ a full matrix algebra, and admissible c.c.p. maps $\psi_n:  A\to F_n\otimes\mathfrak A$ and c.c.p. module maps $\varphi_n: F_n\otimes\mathfrak A\to A$ such that the restriction $\varphi_n^{(i)}$ of $\varphi_n$ to $F_n^{(i)}$ is order 0 module map, for $0\leq i\leq d$, such that $\varphi_n\circ\psi_n\to {\rm id}_A$ in point-norm, and for the induced map $\Psi:=(\psi_n): A\to \prod_n F_n\otimes\mathfrak A$, $q\circ\Psi: A\to \prod_n (F_n\otimes\mathfrak A)/\bigoplus_n (F_n\otimes\mathfrak A)$ is a $*$-homomorphism module map, where $q$ is the corresponding quotient map. let $\tau: A\to\mathfrak A$ be an $\mathfrak A$-quasitrace. By Lemma \ref{lin}$(ii)$, $\tau\circ\varphi_n^{(i)}: F_n^{(i)}\otimes \mathfrak A\to \mathfrak A$ is an $\mathfrak A$-quasitrace, which is then an $\mathfrak A$-trace, by Lemma \ref{qt}. Since $\sup_n\|\tau\circ\varphi_n^{(i)}\|<\infty$, for a free unltrafilter $\omega$ on $\mathbb N$, we get an $\mathfrak A$-trace  $\tau_\omega^{(i)}: \prod_n (F_n\otimes\mathfrak A)/\bigoplus_n (F_n\otimes\mathfrak A)\to \mathfrak A$, defined by, $$\tau_\omega^{(i)}(q(x)):=\lim_{n\to\omega}\tau(\varphi_n^{(i)}(x_n));\ \ \ x:=(x_n)\in \prod_n (F_n\otimes\mathfrak A),$$
for $0\leq i\leq d$, which exists since $\sup_{a\in A_{+,1}} \tau(a)\leq \sup_{a\in A_{+,1}} D_\tau(a)<\infty$. For the $*$-homomorphism module map $\Psi^{(i)}:=(\psi_n^{(i)}): A\to \prod_n F_n^{(i)}\otimes\mathfrak A$, $$\tau_i:=\tau_\omega^{(i)}\circ q\circ \Psi^{(i)}: A\to\mathfrak A$$ is an $\mathfrak A$-trace. By Lemma \ref{sum}, 
\begin{align*}
\tau(a)&=\lim_{n\to\omega} \tau(\varphi_n\circ\psi_n(a))\\&\leq 2^n\sum_{i=0}^{d}\lim_{n\to\omega} \tau(\varphi_n^{(i)}\circ\psi^{(i)}_n(a))\\&=2^n\sum_{i=0}^{d}\lim_{n\to\omega}\tau_\omega^{(i)}\circ q\circ \Psi^{(i)}\\&=2^n\sum_{i=0}^{d}\tau_i(a)=:\tau_0(a),
\end{align*}
where $\tau_0$ is an $\mathfrak A$-trace. Therefore, $\tau$ is an $\mathfrak A$-trace, by Lemma \ref{lin}$(i)$.
\end{proof}

\begin{remark}
Under the weaker condition ${\rm dim}_{\rm nuc}^\mathfrak A(A)<\infty$, if the upward maps $\varphi_n$ in Definition \ref{nuclear} could be taken to be c.p. module maps (not just admissible maps), then an argument similar to that of used in the proof of Lemma \ref{key} shows that the downward maps $\psi_n$ could be taken to be almost order 0 in the sense of \cite{w}. The above argument could then be repeated with the difference that now $\Psi$ is only an order 0 module map. The fact that each $\tau_i$ is an $\mathfrak A$-trace now follows from Lemma \ref{lin}$(ii)$, and the proof could be continued, showing that in this weaker case, if one could arrange for upward module maps, we still get $\mathcal{QT}_\mathfrak A(A)=\mathcal{T}_\mathfrak A(A)$.   
\end{remark}

We say that $a,b\in A$ are orthogonal, writing $a\perp b$, if multiplication of any of $a$ or $a^*$ to $b$ or $b^*$ is zero. 

\begin{definition}
Let $A$ be unital and $\mathfrak A$ be unital and monotone complete. An $\mathfrak A$-rank function is a function $D: A\to \mathfrak A_{+,1}$ satisfying,

(1) $\sup D= 1$,

(2) If $a\perp b$ then $D(a+b)=
D(a) + D(b)$,

(3) $D(a) = D(aa^*) = D(a^*a) = D(a^*)$, for $a\in A$,

(4) If $0\leq a \leq b$, then $D(a)\leq D(b)$,

(5) If $a\lesssim b$ ($a$ is Cuntz-subequivalent to $b$), then $D(a)\leq D(b)$,

\noindent A rank function is subadditive if, 

(6) $D(a + b) \leq D(a) + D(b)$, for $a,b\in A$,

\noindent   weakly subudditive if, 

(6') $D(a + b) \leq D(a) + D(b)$, for $a,b\in A_+$ with $ab=ba$,

\noindent and lower semicontinuous (l.s.c.) if,

(7)  whenever $a_n\to a$ in norm, for each $\varepsilon>0$, $(1-\varepsilon)D(a)\leq D(a_n)$, for large $n$. 
\end{definition}

When $A$ is not unital, one could replace (1) by (1)$^{'}$ $\sup D<\infty$ (which is then automatic: if there is a sequence $(a_n)\subseteq A_+$ with $D(a_n) \geq n1_\mathfrak A$ and $\|a_n\|< 2^{-n}$, then $\|D(\Sigma_n a_n)\|\geq n$, for each $n\geq 1$, which is absurd). 
We denote the set of subadditive lower semicontinuous  $\mathfrak A$-rank functions on $A$ by $\mathcal{RF}_\mathfrak A (A)$. 

Given $\varepsilon>0$, let $f_\varepsilon$ be the continuous function on $\mathbb R$ which
is zero on $(-\infty,\varepsilon/2]$, linear on $[\varepsilon/2,\varepsilon]$, and one on $[a, +\infty)$.
A proof almost identical to that of \cite[Proposition I.1.5]{bh} gives the following result. 

\begin{lemma} [Blackadar-Handelman]  \label{lsc}
Let $A$ be unital and $\mathfrak A$ be unital and monotone complete. For an $\mathfrak A$-rank function  $D: A\to \mathfrak A_{+,1}$, the following are equivalent:

$(i)$ $D$ is l.s.c.,

$(ii)$ $D(a)=\sup_{\varepsilon>0} D(f_\varepsilon(a)).$
\end{lemma}	

We adapt the proof of \cite[Theorem II.1.7]{bh} to get the following result.

\begin{proposition} \label{qt2}
Let  $A$ and $\mathfrak A$  be  finite AW*-algebra and tr$_A: A\to Z(A)$ and tr$_\mathfrak A: \mathfrak A\to\mathfrak Z$ be  center-valued quasitraces on $A$ and $\mathfrak A$. Assume that $A$ is weakly central as a $\mathfrak A$-module. Then $\tau=\phi\circ {\rm tr}_A$ defines a one-to-one correspondence between $\mathfrak Z$-valued $\mathfrak Z$-quasitraces $\tau\in\mathcal {QT}_\mathfrak Z(A)$ and c.p. module maps $\phi\in CP_\mathfrak Z(Z(A), \mathfrak Z)$. 	
\end{proposition}
 \begin{proof}
 	Let $\tau\in \mathcal{QT}_\mathfrak Z(A)$ and let $\phi_0$ be the restriction of $\tau$ on $Z(A)$. Then $\phi:={\rm tr}_\mathfrak A\circ\phi_0$ is a c.p.  map, since $\tau$ is linear and positive on $Z(A)$, and positive linear maps into commutative C*-algebras are automatically c.p. This is also automatically a $\mathfrak Z$-module map, since given $\alpha\in \mathfrak Z$,  
 $$\phi(a\cdot\alpha)={\rm tr}_\mathfrak A(\tau(a\cdot\alpha))={\rm tr}_\mathfrak A(\tau(a)\alpha)={\rm tr}_\mathfrak A(\tau(a))\alpha,$$
 for $a\in A$, as ${\rm tr}_\mathfrak A$ is a $\mathfrak Z$-module map.
 		
 	Put $\tau_1:= \phi\circ {\rm tr}_A$ and observe that $\tau_1(e) = \tau(e)$,
 	for every central projection $e\in A$, and so for every simple projection $e\in A$ (i.e., a projection $e$ which divides its central cover; c.f., \cite[page 157]{be}). Since the spectrum of $Z(A)$ is totally disconnected, one may find finite
 	orthogonal families $\{q_i\}$ and $\{r_j\}$ of simple projections in $A$  with $\Sigma_i q_i \leq p \leq \Sigma_j r_j$ and ${\rm tr}_A(\Sigma_j r_j-\Sigma_i q_i) \leq \varepsilon 1_A$. Since both $\tau(p)$ and $\tau_1(p)$ are between $\tau_1(\Sigma_i q_i)$ and $\tau_1(\Sigma_j r_j)$ in $\mathfrak A_+$, it follows that the self-adjoint element $\alpha:= \tau(p)-\tau_1(p)$ lies in between $-\beta$ and $\beta$, for positive $\beta:={\rm tr}_A(\Sigma_j r_j-\Sigma_i q_i)\leq\varepsilon 1_A$, and  we have $0\leq \beta-\alpha\leq 2\beta$. Hence, 
 	$$\|\tau(p)-\tau_1(p)\|=\|\alpha\|\leq \|\beta\|+\|\beta-\alpha\|\leq 3\|\beta\|\leq 3\varepsilon,$$
 	and since $\varepsilon>0$ was arbitrary, $\tau=\tau_1$ on
projections, and so everywhere by continuity. 

Conversely, if $\phi\in CP_\mathfrak Z(Z(A), \mathfrak Z)$, then $\tau:=\phi\circ {\rm tr}_A$ clearly has all the conditions of a $\mathfrak Z$-quasitrace, except the module map property. Now, by assumption we have, 
$1_A\cdot\alpha\in Z(A)$, for each $\alpha\in\mathfrak Z$, thus,
 $$\tau((a\cdot\alpha))=\phi\big({\rm tr}_A(a(1_A\cdot\alpha))\big)=\phi\big({\rm tr}_A(a)(1_A\cdot\alpha)\big)=\phi\big({\rm tr}_A(a)\cdot\alpha\big)=\phi\big({\rm tr}_A(a)\big)\cdot\alpha,$$
for $a\in A$, as ${\rm tr}_A$ is a $Z(A)$-module map.

Finally, if $\phi_1\circ{\rm tr}_A=\phi_1\circ{\rm tr}_A$, then $\phi_1=\phi_2$, as ${\rm tr}_A=$id on $Z(A)$.  
 \end{proof}

The next result extends \cite[Theorem 11.2.2]{bh}.

\begin{theorem} \label{main}
Let $A$ be unital and $\mathfrak A$ be an AW*-algebra. There is an affine bijection: $\mathcal{QT}_\mathfrak A(A)\to \mathcal{RF}_\mathfrak A(A)$ with continuous inverse.
\end{theorem}
 \begin{proof}
We may assume that $A$ is unital. Given $D\in \mathcal{RF}_\mathfrak A(A)$, $f\in \mathfrak A^*$, and commutative unital C*-subalgebra $B\simeq C(X)$ of $A$, there is a countably additive probability measure on the algebra generated by the $\sigma$-compact open subsets
of $X$ by $\mu_{{}_{D,f}}(U):=f(D(g))$, for $g\in C(X)$ with $0\leq g\leq 1$ and ${\rm supp}(g)=U$. Define $\tau_D: B\to \mathfrak A^{**}$ by
$$ \langle \tau_D(b),f\rangle:=\int_X \hat b d\mu_{{}_{D,f}}\ \ (b\in B, f\in \mathfrak A^*).$$
where $b\mapsto \hat b$ is the Gelfand transform on $B$. Fix $b\in B$. We claim that $\tau_D(b)$ is weak$^*$-continuous on $\mathfrak A^*$. By Jordan decomposition, we only need to show that if $(f_i)$ is a net in $\mathfrak A^*_+$ with $f_i\to 0$ (weak$^*$), then $\int_X \hat b d\mu_{{}_{D,f_i}}\to 0$. But since each $\mu_{{}_{D,f_i}}$ is a positive measure in this case, $$\int_X \hat b d\mu_{{}_{D,f_i}} \leq \|\hat b\|\mu_{{}_{D,f_i}}(X)=\|\hat b\|f_i(D(1))=\|\hat b\|f_i(1_\mathfrak A)\to 0,$$
as $i\to infty$. This means that the range of $\tau_D$ is inside $\mathfrak A$. Summing up,  $\tau_D(a)\in\mathfrak A$ is defined
unambiguously on normal elements $a\in A$.
For an arbitrary element $a\in A$, write $a = a_1 + ia_2$ with $a_1, a_2$ selfadjoint, and define $\tau_D(a) = \tau_D(a_1) + i\tau_D(a_2)$. Let us observe that $\tau_D(x^*x) = \tau_D(xx^*)$. Let $A\subseteq \mathbb B(H)$ be a faithful representation  on a Hilbert space $H$, and $x = u|x|$ be the polar
decomposition of $x$ with $u\in\mathbb B(H)$. Take a non negative continuous function $h$ on $\sigma(x^*x)\cup\{0\}=\sigma(xx^*)\cup\{0\}$ with $h(0)=0$, and put $y = uh(x^*x)^{\frac{1}{2}}\in A$, then $y^*y =h(x^*x)$ and,
$yy^* =h(xx^*)$. Since, 
$$D(h(xx^*))=D(yy^*)=D(y^*y)=D(h(x^*x)),$$
and C*-subalgebras generated by the sets $\{h(xx^*), 1_A\}$ and $\{h(x^*x), 1_A\}$ are the same, it follows that $\tau_D(h(xx^*))=\tau_D(h(x^*x))$, for each function $h$ as above, and the claim follows. Since $D$ is subadditive, it could be extended it can be extended to an enveloping AW*-algebra of $A$ (as in
\cite[Section I.4]{bh}), therefore, $\tau_D$ is an $\mathfrak A$-quasitrace, by an argument as in the proof of Proposition \ref{qt2}. Finally, if $D_i\to D$ in point-norm, then $\mu_{{}_{D_i,f}}\to \mu_{{}_{D,f}}$, uniformly on $f$, thus, $\tau_{D_i}\to \tau_D$ in point-norm.

Conversely, given $\tau\in\mathcal{QT}_\mathfrak A(A)$,  define $D_\tau(a) =\sup_{\varepsilon>0} \tau(f_\varepsilon(|a|))$. For $f\in \mathcal S(\mathfrak A)$, being a positive linear functional on the C*-subalgebra $B$ generated by $|a|$, $f\circ \tau$ is bounded, say by $M>0$, thus $\tau(f_\varepsilon(|a|))\leq M1_\mathfrak A$, and so the above supremum exists, as masas in $\mathfrak A$ are monotone complete. Since $\tau$ extends to $\mathbb M_2(A)$, so does $D_\tau$. Given $a,b\in A$, let $a_n:=f_{\frac{1}{n}}(aa^*)$ and $a^{'}_n:=f_{\frac{1}{n}}(a^*a)$, and define $b_n, b^{'}_n$ similarly. Then for matrix units $e_{ij}\in \mathbb M_2(A)$ (with $1_A$ at $(i,j)$-position and zero elsewhere), 
$$(a_ne_{11}+b_ne_{12})(ae_{11}+be_{22})(a^{'}_ne_{11}+b^{'}_ne_{21})\to (a+b)e_{11},$$
in norm, as $n\to \infty$, that is,
$$(a+b)e_{11}\preceq ae_{11}+be_{22},$$
therefore, $D_\tau(a+b)=D_\tau((a+b)e_{11})\leq D_\tau(ae_{11}+be_{22})=D_\tau(a)+D_\tau(b),$
i.e., $D_\tau$ is subadditive. Also it follows from Lemma \ref{lsc}, that $D_\tau$ is l.s.c. The rest of properties follow from definition and $D_\tau$ is an $\mathfrak A$-rank function.
\end{proof} 

\begin{corollary}
Let $\mathfrak A$ be an AW*-algebra and $A$ be unital and exact. Then every $\mathfrak A$-rank function $D$ is of the form $D_\tau:=\sup_{\varepsilon>0} \tau(f_\varepsilon(|\cdot|))$, for some $\tau\in\mathcal T_\mathfrak A(A)$.
\end{corollary}
\begin{proof}
This follows from Lemma \ref{qt} and Theorem \ref{main}.  
\end{proof}

Recall that an element $\alpha\in\mathcal A$ is strictly positive, writing $\alpha>0$, if $f(\alpha)>0$, for each $f\in\mathcal S(\mathfrak A)$. For $0\leq a\leq b$, we write $a<b$ if $b-a>0$.  

\begin{definition}
Let $\mathfrak A$ be a $\sigma$-unital C*-algebra. We say that $A$ has $\mathfrak A$-strict comparison of positive elements, or simply, $\mathfrak A$-SC, if for any $n\geq 1$ and $a,b\in\mathbb M_n(A)_+$, if $D_\tau(b)-D_\tau(a)>0$ in $\mathfrak A$, for each $\tau\in\mathcal T_\mathfrak A(A)$, then $a\preceq b$. 
\end{definition}

\begin{remark} \label{sc}
	Let $A$ has at least one non-zero $\mathfrak A$-trace.
	
$(i)$ When $A$ is simple, $\mathfrak A$-SC implies that  $a\preceq b$ holds for $a,b\in A_+$ with $a$ not Cuntz-equivalent to a projection, whenever $D_\tau(a)\leq D_\tau(b)$, for each non-zero $\mathfrak A$-trace $\tau$: Since $a$ is not Cuntz-equivalent to a projection, $0\in\sigma(a)$ is not an isolated point \cite[14.2.12]{s}. Choose a strictly decreasing sequence $(\varepsilon_n)$ of positive reals in $\sigma(a)$ going to 0. Since $\varepsilon_{n+1}<\varepsilon_n$, we have $(a-\varepsilon_{n})_+\preceq(a-\varepsilon_{n+1})_+$, thus, $D_\tau\big((a-\varepsilon_{n})_+\big)\leq D_\tau\big((a-\varepsilon_{n+1})_+\big)$, as $D_\tau$ is an $\mathfrak A$-rank function. Let $f\in \mathcal S(\mathfrak A)$ and $\mu_{{}_{D_\tau,f}}$ be the measure constructed in the proof of Theorem \ref{main}. Let $U_n$ be the set of elements $t\in (\varepsilon_{n+1},\varepsilon_n]$ for which $(t-\varepsilon_{n+1})_+>0$. Choose $g_n\in C(\sigma(a))$ with $g=0$ off $U_n$ and  $g(\varepsilon_{n+1})\neq 0$. Then, $$\mu_{{}_{D_\tau,f}}(U_n)\geq f(D_\tau(g))\geq f(\tau(g))>0,$$
where the last strict inequality follows from the fact that $f\circ\tau$ is non-zero trace on a simple C*-algebra, and so faithful. Now, 
$$f\big(D_\tau((a-\varepsilon_{n+1})_+)- D_\tau((a-\varepsilon_{n})_+)\big)=\mu_{{}_{D_\tau,f}}(U_n)>0,$$
for each $f\in \mathcal S(\mathfrak A)$, which means that, $$D_\tau\big((a-\varepsilon_{n})_+\big)< D_\tau\big((a-\varepsilon_{n+1})_+\big)\leq D_\tau(a)\leq D_\tau(b),$$ in $\mathfrak A$, thus, $(a-\varepsilon_{n})_+\preceq b$, and so $a\preceq b$, by \cite[14.1.8]{s}. 

$(ii)$ When $A$ is simple and exact, $\mathfrak A$-SC implies that  $a\preceq b$ holds for $a$ Cuntz-equivalent to a projection and $b$ not Cuntz-equivalent to any projection, exactly when $D_\tau(a)< D_\tau(b)$, for each non-zero $\mathfrak A$-trace $\tau$: if $a\preceq b$ then for any $0<\varepsilon<1$, there is $\delta>0$ with $(a-\varepsilon)_+\preceq(b-\delta)_+$ \cite[14.1.9]{s}. Without loss of generality, we may assume that $a$ is a projection, thus,  $(a-\varepsilon)_+=\lambda a$, for some $\lambda>0$. Therefore,
\begin{align*}
D_\tau(a)&=\sup_n \tau(a^{\frac{1}{n}})=\sup_n \tau((\lambda a)^{\frac{1}{n}})=D_\tau(\lambda a)\\&=D_\tau\big((a-\varepsilon)_+\big)\leq D_\tau\big((b-\delta)_+\big)\\&<D_\tau\big((b-\delta/2)_+\big)\leq D_\tau(b),	
\end{align*}
where the strict inequality is proved by an argument as in  $(i)$ above. The converse implication is immediate.
\end{remark}

\begin{lemma} \label{asc}
Consider the following properties:

$(i)$ $A$ has SC,

$(ii)$ $A$ has $\mathfrak A$-SC.

\noindent Then $(i)\Rightarrow (ii)$ when $\mathfrak A$ is unital. The converse also holds when  $\mathfrak A$ is an AW$^*$-algebra.
\end{lemma}
\begin{proof}
 $(i)\Rightarrow (ii)$. Given $\tau_0\in\mathcal T(A)$, let $\tilde\tau_0(a):=\tau_0(a)1_\mathfrak A$, then 
 $$D_{\tilde\tau_0}(a)=\sup_{\varepsilon>0}\tilde\tau_0(f_\varepsilon(|a|))=\sup_{\varepsilon>0}\tau_0(f_\varepsilon(|a|))1_\mathfrak A=D_{\tau_0}(a)1_\mathfrak A.$$
Given $a,b\in \mathbb M_n(A_+)$, if $D_\tau(b)-D_\tau(a)>0$, for each $\tau\in\mathcal T_\mathfrak A(A)$, then applying this to $\tau:=\tilde\tau_0$, we get, $D_{\tau_0}(b)-D_{\tau_0}(a)>0$, for each $\tau_0\in\mathcal T(A)$, which implies $a\preceq b$, by assumption.

 $(ii)\Rightarrow (i)$. Given  $\tau\in\mathcal T_\mathfrak A(A)$ and $g\in \mathcal S(\mathfrak A)$, let $\tau_g(a):=g(\tau(a)).$ Then, 
  $$D_{\tau_g}(a)=\sup_{\varepsilon>0}g(\tau(f_\varepsilon(|a|)))=g(\sup_{\varepsilon>0}\tau(f_\varepsilon(|a|))=g\circ D_{\tau}(a).$$ 
  Given $a,b\in \mathbb M_n(A_+)$, if $D_{\tau_0}(b)-D_{\tau_0}(a)>0$, for each $\tau_0\in\mathcal T(A)$, then applying this to $\tau_0:=g\circ \tau$, we get, $g(D_{\tau}(b)-D_{\tau}(a))>0$, for each $g\in\mathcal S(\mathfrak A)$, which means that $D_{\tau}(b)-D_{\tau}(a)$ is strictly positive in $\mathfrak A$, which in turn  implies $a\preceq b$, by assumption.   
\end{proof}

Next, let us explore the relation to Cuntz semigroup. Let $\mathbb M_\infty(A)$ be the union of matrix algebras $\mathbb M_n(A)$ for $n\geq 1$. For $a\in \mathbb M_n(A), b\in \mathbb M_m(A)$, let $a\oplus b\in\mathbb M_{m+n}(A)$ be the diagonal matrix with diagonals $a$ and $b$. Given $a,b\in\mathbb M_\infty(A)_+$, we may assume that $a,b\in\mathbb M_n(A)_+$, for some $n$ (simply by adding a zero block to the one with smaller size). We could then write $a\preceq b$ if this happens in $\mathbb M_n(A)$ (and this is clearly independent of the choice of $n$). We write $[a]$ to denote the Cuntz equivalence class of $a$ and equip $W(A):=\mathbb M_\infty(A)_+/\sim$ with $[a]+[b]:=[a\oplus b]$, and call it the Cuntz semigroup of $A$. Similarly, $Cu(A):=W(A\otimes \mathbb K)$ is called the stablised Cuntz semigroup of $A$. 

\begin{proposition} \label{state}
	Let $A$ and $\mathfrak A$ be unital. For each state $f\in \mathcal S(\mathfrak A)$ and each $\mathfrak A$-rank function $D$ on $A$, 
	$$d: W(A)\to [0,\infty);\ \  d([a]):=f(D(a)),\ \ \ (a\in A_+),$$
	is a state of the Cuntz semigroup. Moreover, these exhaust the state space of $W(A)$. 
	
	If moreover $\mathfrak A$ is an AW*-algebra (and $A$ is exact), then every state on $W(A)$ is of the form $f\circ D_\tau$, for some state $f$ on $\mathfrak A$ and some $\mathfrak A$-quasitrace ($\mathfrak A$-trace) $\tau$ on $A$.
\end{proposition}
\begin{proof}
	Given   $f\in \mathcal S(\mathfrak A)$ and $D\in \mathcal{RF}_\mathfrak A(A)$, let $d([a]):=f(D(a)),$ for $a\in A_+$, then $d([1_A])=f(D(1_A))=f(1_\mathfrak A)=1,$ and $$d([a]+[b])=f(D(a\oplus b))=f(D(a)+D(b))=d([a])+d([b]),$$ thus, $d$ is a state of the scaled semigroup $W(A)$. Conversely, given a state $d$ of $W(A)$, let $D(a):=d([a^*a])1_\mathfrak A,$ then $D$ satisfies conditions of an $\mathfrak A$-rank function by properties of Cuntz-subequivalence; c.f., \cite[14.1.5]{s}), and for each state $f\in \mathcal S(\mathfrak A)$, and $a\in A_+$, 
	$$f(D(a))=d([a^*a])f(1_\mathfrak A)=d([a^2])=d([a]).$$
	The last statement follows from Theorem \ref{main} and Lemma \ref{qt}.    
\end{proof}

Recall that an additive partially ordered monoid $S$ is {\it weakly unperforated} if $nx\neq 0$ for some $n$ implies $x\neq 0$, for each $x\in S$, and {\it almost unperforated} if $nx\leq my$ for some $n>m$ implies $x\leq y$, for each $x,y\in S$. An (abelian) ordered group $(G, G_+)$ is weakly unperforated if its positive cone $G_+$ is weakly unperforated. The adjective ``weakly'' suggests that we  have also a stronger version. Indeed,  an (abelian) ordered group $(G, G_+)$ is called {\it unperforated} if $nx\geq 0$ for some $n$ implies $x\geq 0$, for each $x\in G$.
An unperforated group must be torsion-free, while a weakly unperforated group can 
have torsion, $\mathbb Z \oplus\mathbb Z_2$, with strict ordering from the first coordinate, being an example. 
 
An element $u\in G_+$ is called an {\it order unit} if for any $x\in G$ there is $n > 0$ with $x < nu$. Also, $G$ is called {\it simple} if every nonzero positive element is an order unit.  A  simple weakly unperforated scaled ordered 
group $(G,G_+,u)$ is known to have strict ordering from its states, i.e., 
$x\in G_+\backslash\{0\}$ iff $d(x) > 0 $ for state $d\in\mathcal S(G)$ \cite[Theorem 6.8.5]{bla}. 
 
\begin{corollary} \label{wup}
Let $\mathfrak A$ be an AW*-algebra and $A$ be exact and unital. If $W(A)$ is weakly unperforated, then $A$ has $\mathfrak A$-SC. The same holds when $W(A)$ is almost unperforated, if  $A$ is moreover simple.
\end{corollary}
\begin{proof}
Assume that $a,b\in \mathbb M_\infty(A)_+$ and $D_\tau(b)-D_\tau(a)$ is an strictly positive element in $\mathfrak A$. Then $f(D_\tau(a))<f(D_\tau(b))$, for each $f\in \mathcal S(\mathfrak A)$. It follows from Proposition \ref{state} that $d([a])<d([b])$, for each $d\in\mathcal S(W(A))$. 

If $W(A)$ is weakly unperforated  then the Grothendieck group of $W(A)$ has strict ordering from its states by \cite[Theorem 6.8.5]{bla}, thus, $[a]\leq [b],$ i.e., $a\preceq b$, as required. If $A$ is simple then so is its Cuntz semigroup $W(A)$. In this case, if $W(A)$ is weakly unperforated then $W(A)$ has strict ordering from its states by \cite[Proposition 3.2]{ro}, therefore again we have $[a]\leq [b],$ i.e., $a\preceq b$.  
\end{proof}

When $A$ is a unital exact C*-algebra, by Lemma \ref{qt}, $\mathcal{QT}_\mathfrak A(A)=\mathcal{T}_\mathfrak A(A)$, which is  convex, and is point-ultra-strong compact set, if moreover $\mathfrak A$ is a von Neumann algebra \cite[1.3.7]{bo}.

Let $\Delta$ be a (compact) convex subset of a topological vector space, and let 
${\rm LAff}_\mathfrak A(\Delta)_+$ denote the set of all l.s.c., affine maps $f: \Delta\to \mathfrak A_+$. Here being lower semicontinuous (l.s.c.) means that whenever $t_n\to t$ in $\Delta$, for each $\varepsilon>0$, $(1-\varepsilon)f(a)\leq f(a_n)$, for large $n$. The set of continuous elements in ${\rm LAff}_\mathfrak A(\Delta)_+$ is denoted by ${\rm Aff}_\mathfrak A(\Delta)_+$. We say that $f$ is bounded, if there is $\alpha\in\mathfrak A_+$ with $f(t)\leq \alpha$, for each $t\in\Delta$, and denote bounded part of ${\rm LAff}_\mathfrak A(\Delta)_+$ by ${\rm LAff}^b_\mathfrak A(\Delta)_+$. Finally,  an element $f$ is said to be strictly positive if $f(t)$ is strictly positive in $\mathfrak A$, for each $t$ (when $\mathfrak A$ is unital, the constant function 1 being an example. When $\mathfrak A$ is $\sigma$-unital with a strictly positive element $h$, the constant function with constant value $h$ is an example). We denote the strict positive parts of the above defined sets by ${\rm LAff}_\mathfrak A(\Delta)_{++}$ and ${\rm LAff}^b_\mathfrak A(\Delta)_{++}$, respectively. The corresponding sets are also defined for Aff instead of LAff.  
					
For a projection $p\in \mathbb M_\infty(A)$, consider the Murray-von Neumann equivalence class $\langle p\rangle$ of $p$ (which is not necessarily the same as its Cuntz equivalence class $[p]$). Then $\langle p\rangle\mapsto [p]$ is a well-defined semigroup morphism from the semigroup $V(A)$ generating $K_0(A)$ (as it Grothendieck group) to the Cuntz semigroup, which is also injective, when $A$ is stably finite (in the sense of Kaplansky), and $W(A)=V(A)\sqcup W(A)_+$, where the last set consists of Cuntz classes of purely positive elements (i.e., those not Cuntz equivalebnt to any projection) \cite[14.2.8]{s}.

Define $\hat p\in  {\rm LAff}^b_\mathfrak A(\mathcal{T}_\mathfrak A(A))_{++}$ by $\hat p(\tau)=\tau(p)$, for $\tau\in\mathcal{T}_\mathfrak A(A)$. 

\begin{definition}
The semigroup $W_\mathfrak A(A):=V(A)\sqcup {\rm LAff}^b_\mathfrak A(\mathcal{T}_\mathfrak A(A))_{++}$, where  $V(A)$ is the copy of the semigroup of Cuntz equivalence classes of projections under the map $[p]\mapsto \hat p$, is called the $\mathfrak A$-Cuntz semigroup of $A$.
\end{definition} 
			
Note a slight abuse of notation here: when $\mathfrak A=\mathbb C$, we have $W_\mathbb C(A)=\tilde W(A)$ which is isomorphic to $W(A)$  for simple, separable, unital, nuclear C*-algebras with almost unperforated and weakly divisible Cuntz semigroup \cite[14.3.38]{s}.

Define, 

$$\iota: W(A)_+\to {\rm LAff}^b_\mathfrak A(\mathcal{T}_\mathfrak A(A))_{++}; \ \ \  [a]\mapsto (\tau\mapsto D_\tau(a)), \ \ (a\in A_+).$$

This is well-defined, since $D_\tau$ is constant on Cuntz classes.  

\begin{lemma}
Let $A$ be unital and separable with at least one non-zero $\mathfrak A$-trace. 

$(i)$ $\iota$ is a semigroup morphism, if $A$ is simple with stable rank one, 

$(ii)$ $\iota$ is an order embedding, if moreover $A$ has $\mathfrak A$-SC. If $A$ is also exact, $\iota$ induces an order embedding $\tilde\iota: W(A)\to W_\mathfrak A(A)$, with $\tilde\iota={\rm id}$ on $V(A)$ and $\tilde\iota=\iota$ on $W(A)_+$.   
\end{lemma}  			
\begin{proof}
$(i)$ Under the assumption, $W(A)_+$ is a subsemigroup of $W(A)$ \cite[14.2.13]{s}. Given $a\in A_+$, $f\in \mathcal S(\mathfrak A)$, and $\tau\in\mathcal{T}_\mathfrak A(A)$, if $f(\iota([a])(\tau))=f(D_\tau(a))=0$, then $a=0$, as $f\circ D_\tau$ is faithful by simplicity of $A$, but this could not be the case as $a$ is purely positive, thus $\iota([a])(\tau)>0$ in $\mathfrak A$, i.e., $\iota$ ranges in ${\rm LAff}^b_\mathfrak A(\mathcal{T}_\mathfrak A(A))_{++}$. The fact that $\iota$ is a semigroup morphism  follows from the fact that $D_\tau(a\oplus b)=D_\tau(a)+D_\tau(b),$ for $a,b\in A_+$.

$(ii)$ If $[a]\leq [b]$, then $a\preceq b$, which is equivalent to $D_\tau(a)\leq D_\tau(b)$ by Remark \ref{sc}$(i)$. The last assertion now follows from Remark \ref{sc}$(ii)$.
\end{proof}

\begin{proposition}
	Let $A$ be unital and a $*$-module. Let $Z$ be the center of the maximal finite summand $M$ of $A^{**}$, then there is an affine isometric isomorphism from $\mathcal T_\mathfrak A(A)$ onto $CP_\mathfrak A(Z, \mathfrak A^{**})$, consisting of linear, normal, c.p. module maps from $Z$ into $\mathfrak A^{**}$. When $\mathfrak A$ is a von Neumann algebra, this could be arranged onto $CP_\mathfrak A(Z, \mathfrak A)$.
\end{proposition}
\begin{proof}
	Let $A^{**}=M\oplus N$, with $N$ infinite. The units $1_M$ and $1_N$ are orthogonal projections in $A^{**}$, with $1_M+1_N=1_{A^{**}}:=1$, which are finite and purely infinite, respectively (by maximality of $M$). Given a projection $\alpha\in \mathfrak A$, 
	$$(1\cdot\alpha)^2=(1\cdot\alpha)(1\cdot \alpha)=((1\cdot\alpha)1)\cdot\alpha=1.\cdot\alpha^2=1\cdot\alpha,$$ 
	and $(1.\cdot\alpha)^*=1\cdot\alpha^*=1\cdot\alpha,$
	that is, $1\cdot\alpha$ is a projection in $A^{**}$. Multiplying both sides of
	$1\cdot \alpha=1_M\cdot\alpha+1_N\cdot\alpha$ from left by $1_N$, we get $1_N(1_M\cdot\alpha)=0$, that is, $1_M\cdot\alpha\in M$. Thus $$y\cdot\alpha=(y1_M\cdot\alpha)=y(1_M\cdot\alpha)\in M, \ \ (y\in M),$$
	for each projection $\alpha$ (and so, for each $\alpha\in \mathfrak A$, as $\mathfrak A$ is generated in norm by its set of projections), i.e., $M$ is a submodule of $A^{**}$. The same then holds for $N$. Let $Z:=Z(M)$ be the center of $M$, then for selfadjoint elements $z\in Z$, $y\in M$ and $\alpha\in \mathfrak A$, we have,
	$$(z\cdot\alpha)y=\big(y^*(z\cdot\alpha)^*\big)^*=\big(y(z^*\cdot\alpha^*\big)^*=(z\cdot\alpha)y,$$
	thus, $Z$ is a submodule of $M$ (as all these algebras are generated by their selfadjoint elements).
	
	 Let ${\rm tr}_M: M\to Z(M)=:Z$ be the canonical center-valued trace on $M$ \cite[Theorem 2.6]{t}. Then $\phi\circ {\rm tr}_M=\phi$, for each finite tracial state $\phi$ on $M$. Let $q: A^{**}\to M$ be the canonical projection. Given $\tau\in \mathcal T_\mathfrak A(A)$,  consider the normal extension $\tau^{**}: A^{**}\to \mathfrak A^{**}$, Then $\tau^{**}=0$ on $N$, by maximality of $M$, in particular, for each $f\in \mathcal S(\mathfrak A)$, $\phi:=f\circ\tau^{**}$ is a tracial state on $M$, as $\tau(1_A)=1_\mathfrak A$. Therefore,  $\tau^{**}\circ {\rm tr}_M=\tau^{**}$, as states separate the points of $\mathfrak A$. Identifying $A$ with its image in $A^{**}$, we may rewrite the last equality as, $\tau^{**}({\rm tr}_M(q(a))=\tau(a)$, for $a\in A$. 
	
	Define, 
	$$\iota: \mathcal T_\mathfrak A(A)\to CP_\mathfrak A(M, \mathfrak A^{**}); \ \ \iota(\tau)(x):=\tau^{**}(x), \ \ (x\in M),$$
	which is normal, since $M_*$ could be identified with the set of  elements in $A^*$ vanishing on $N$, and so it contains $f\circ\tau^{**}$, for each $f\in \mathfrak A^*$. Also, $\iota$ is injective, since for $\tau\in \mathcal T_\mathfrak A(A)$ with  $\iota(\tau)=0$, $\tau^{**}=0$ on $M$ and so on $A^{**}$ (as $\tau^{**}=0$ on $N$), thus,  $\tau=0$. 
	
	Now, let $j: M_*\to Z_*$ be the restriction map, and observe that  $$j\circ\iota: \mathcal T_\mathfrak A(A)\to CP_\mathfrak A(Z, \mathfrak A^{**})$$ is surjective: given $\phi\in CP_\mathfrak A(Z, \mathfrak A^{**})$, define $\tau_\phi: A\to \mathfrak A$ by $\tau_\phi(a):=\langle {\rm tr}_M(q(a)), \phi\rangle$, $a\in A$, then $\tau_\phi^{**}(x):=\langle {\rm tr}_M(q(x)), \phi\rangle$, $x\in A^{**}$, thus, 
	$$\iota(\tau_\phi)(x)= \tau^{**}_\phi(x)=\langle {\rm tr}_M(q(x)), \phi\rangle = \langle {\rm tr}_M(q(x)), \phi\rangle=\langle x, \phi\rangle,$$
	for $x\in Z$, that is, $j\circ\iota(\tau_\phi)=\phi$, as required.
	
	Next, let us observe that $\j\circ\iota$ is isometric: given $\tau\in \mathcal T_\mathfrak A(A)$,
	\begin{align*}
		\|j\circ\iota(\tau)\|&=\sup_{x\in Z_1}\|\iota(\tau)(x)\|\leq\sup_{x\in M_1}\|\tau^{**}(x)\|\\&\leq\sup_{x\in A^*_1}\|\tau^{**}(x)\|=\sup_{a\in A_1}\|\tau(a)\|\\&=\sup_{a\in A_1}\|\tau^{**}({\rm tr}_M(q(a)))\|\\&\leq\sup_{x\in Z_1}\|\tau^{**}(x))\|\\&=\|j\circ\iota(\tau)\|,
	\end{align*}
	therefore, $\|j\circ\iota(\tau)\|=\|\tau\|$. 
	
	When $\mathfrak A$ is a von Neumann algebra, $\tau^*: \mathfrak A^*\to A^*$ could be restricted to $\mathfrak A_*$, and taking adjoint, we get a normal extension $\tau^{**}: A^*\to\mathfrak A$. The rest goes as above. 
\end{proof}

\begin{remark}
$(i)$ A positive linear map from $Z$ to $\mathfrak A$ is automatically c.p. \cite[2.2.6]{li}, that is, $CP_\mathfrak A(Z, \mathfrak A)=P_\mathfrak A(Z, \mathfrak A)$. 

$(ii)$ In the above lemma, we may replace $\mathfrak A$ with $\mathfrak Z:=Z(\mathfrak A)$ and identify the convex set $\mathcal T_\mathfrak Z(A)$ of $\mathfrak Z$-valued traces on $A$ with $CP_\mathfrak Z(Z, \mathfrak Z)=P_\mathfrak Z(Z, \mathfrak Z)$.
\end{remark}

\begin{corollary}
	Let $A$ be unital and let $M$ be the maximal finite summand of $A^{**}$.  Then there is a continuous affine isomorphism between $Z_+$ and a convex subset of ${\rm Aff}^b_\mathfrak A(\mathcal T_\mathfrak A(A))_+$, where $Z$ is the center of the maximal finite summand of $A^{**}$.
\end{corollary}	
\begin{proof}
Consider the affine map,
$$x\in Z_+\mapsto \hat x\in {\rm Aff}^b_\mathfrak A(\mathcal T_\mathfrak A(A))_+; \ \ \hat x(\tau)=\tau^{**}(x), $$
which is well-defined, since, $$\hat x(\tau)=\tau^{**}(x)\leq \|x\|\tau^{**}(1_{A^{**}})=\|x\|1_\mathfrak A,$$
and $\hat x$ is clearly (point-ultra-strong)-(ultra-strong)-continuous. Also, in the notations of the above proposition, 
$$\|\hat x\|=\sup_{\|\tau\|\leq 1}\|\langle j\circ\iota(\tau), x\rangle\|\leq \|x\|,$$
as $j\circ\iota$ is isometric. 
\end{proof}

\begin{remark}
$(i)$ In the above result, if  moreover for each $x\in Z_+$, there is $\phi\in CP_\mathfrak A(Z, \mathfrak A)$ with $\|\phi\|=1$ and $\|\phi(x)\|=\|x\|$, then the affine isomorphism $x\mapsto \hat x$ is isometric: in the notations of the above proposition,  $\|\tau_\phi\|=\|j\circ\iota(\tau_\phi)\|=\|\phi\|=1$, therefore,
$$\|x\|=\|\phi(x)\|=\|\tau_\phi^{**} (x)\|=\|\hat x(\tau_\phi)\|\leq \|\hat x\|,$$
and so, $\|x\|=\|\hat x\|$, as claimed.

$(ii)$ Unlike the classical case \cite[14.3.33]{s}, the above affine isomorphism is not generally surjective in the module case. Indeed, the surjectivity follows in the classical case from a result of Cuntz-Pedersen, stating that the dual of ${\rm span}_\mathbb R(\mathcal T(A))$ with induced weak$^*$-topology is $A_{sa}^q$ \cite[Proposition 2.7]{cp}. However, this does not seem to have an analog in the module case.

$(iii)$ The lack of surjectivity in the above result has some consequences. In the classical case, the surjectivity of this map is used to show that the map $\tilde\iota$ is an order isomorphism between $W(A)$ and $\tilde W(A)$ when $A$ is simple, separable, unital, exact C*-algebra of stable rank one, satisfying strict comparison of positive elements, and having a weakly divisible Cuntz semigroup (c.f., \cite[14.3.38]{s}). In the module case, $\tilde\iota: W(A)\to W_\mathfrak A(A)$ may fail to be surjective, even if $A$ has $\mathfrak A$-SC and $W(A)$ is weakly divisible. However, it is not hard to characterize the image of $W(A)$ inside $W_\mathfrak A(A)$.          
\end{remark}

\section{Examples} \label{5}

In this section we calculate the module nuclear dimension in some concrete classes of C*-algebras with canonical module structure.

\begin{example}
As the first series of examples, we consider the case that $\mathfrak A=C_0(X)$ is commutative, and $A$ is a $C_0(X)$-algebra \cite[Definition C.1]{wi}. Then there is a $*$-homomorphism $\phi: C_0(X)\to Z(M(A))$, defining a right action of $C_0(X)$ on $A$ by $a\cdot f:=\phi(f)a$. Here are some immediate instances, where module nuclear dimension could be calculated (or at least estimated):

$(i)$ If $Y$ is another locally compact space and $A:=C_0(X\times Y)$, then we have a canonical map
$$\phi: C_0(X)\to C_b(X\times Y); \ \phi(f)(x,y)=f(x),\ \ (x\in X, y\in Y),$$ and when both $X$ and $Y$ are second countable, by Corollary \ref{ub},
$${\rm dim}_{\rm nuc}^{C_0(X)}C_0(X\times Y)\geq \frac{{\rm dim}(X)+{\rm dim}(Y)+1}{{\rm dim}(X)+1}-1.$$
On the other hand, it is easy to see that,
${\rm dim}_{\rm nuc}^\mathfrak A(\mathfrak A\otimes B)\leq {\rm dim}_{\rm nuc}(B)$, thus, 
 $$\frac{{\rm dim}(Y)}{{\rm dim}(X)+1}\leq{\rm dim}_{\rm nuc}^{C_0(X)}C_0(X\times Y)\leq {\rm dim}(Y).$$

$(ii)$ As a concrete example, it follows that $\frac{n}{2}\leq {\rm dim}_{\rm nuc}^{C_0(\mathbb T)}C_0(\mathbb T^n)\leq n$. In particular, ${\rm dim}_{\rm nuc}^{C_0(\mathbb T)}C_0(\mathbb T^2)=1$.    
\end{example}

\begin{example}
Let $A$ be a  simple, separable, infinite dimensional, unital C*-algebra and $\mathfrak A= C((0, 1],A)$ is the cone of $A$. Then $\mathfrak A$ is a continuous $C((0, 1])$-algebra with simple fibers $A$ at each point. Consider $A$ as a right $\mathfrak A$-module via,
$$a\cdot (f\otimes b):=f(1)ab, \ \ (a,b\in A),$$
Then the identity ${\rm id}_A$ on $A$ decomposes via $*$-homomorphisms   
$$A \to  C((0, 1],A)\to A; \ \ a\mapsto {\rm id}\otimes a, \ f\otimes a\mapsto f(1)a,$$ and so ${\rm nd}_\mathfrak A A=0$ (c.f., \cite[Remark 4.3.10]{g}. However, we show that ${\rm dim}_{\rm nuc}^\mathfrak A(A)=\infty$, in this case: If $\psi: A\to \mathfrak A$ is a module map, then,
$$\psi(a\cdot(f\otimes b))=\psi(a)(f\otimes b),$$
that is, $f(1)\psi(ab)=\psi(a)(f\otimes b)$, for $f\in C_0(0,1]$, and $a,b\in A$. When $f(1)\neq 0$, putting $a=1$, we get, $\psi(b)=\frac{\psi(1_A)}{f(1)}(f\otimes b)$. Calculating $\psi(b)(t)$, at $b=1_A$ for $f_\varepsilon(1)=1, f_\varepsilon(t)=2$, for $\varepsilon<t<1-\varepsilon$,
 we get $\psi(1_A)(t)=0$, for $t$ as above, and each $0<\varepsilon<\frac{1}{2}$. By continuity, $\psi(1_A)=0$, and so $\psi=0$, i.e., there is no non-zero module map: $A\to \mathfrak A$. It follows that there is also no non-zero module map: $A\to \mathbb M_n(\mathfrak A)$, for each $n\geq 1$. In particular, the only possible c.c.p. admissible map: $A\to \mathbb M_n(\mathfrak A)$ is of the form $a\mapsto \rho(a)u$, for $\rho\in A^*_{1,+}$ and $u\in \mathbb M_n(\mathfrak A)_+$. Choose distinct elements $a,b\in A$ with $\rho(a)=\rho(b)=1$ (which is possible, as $A$ is infinite dimensional), and for  $\mathcal F:=\{a,b\}$, let $\varphi_n: \mathbb M_{k(n)}(\mathfrak A)\to A$ be an admissible map with $\varphi_n(u)=\varphi_n(\rho(x)u)\approx_{\frac{1}{n}} x$, for $x=a,b$, and some $u\in \mathbb M_{k(n)}(\mathfrak A)_+$, thus, $a\approx_{\frac{2}{n}}b$, for each $n\geq 1$, which is a contraction.            

Similarly, ${\rm dr}_\mathfrak A(A)=\infty$. When $\mathfrak A$ is unital and simple, there is an alternative way to see this using Lemma \ref{key}: there is no  approximately multiplicative sequence of maps $\varphi_n: A\to \mathbb M_{k(n)}(\mathfrak A)$, since then the unit $1_A$ is asymptotically mapped to a projection, but this is not possible,  since $\mathfrak A:=C((0, 1],A)$ does not contain any projection, and
nor does any of the algebras $\mathbb M_{k(n)}(\mathfrak A)$, since any such projection has to be
unitarily equivalent to a diagonal one. By the same token, any of these algebras could  not contain ``almost projections''.
\end{example}

\begin{example} \label{equi}
$(i)$ We give a class of examples in which the implication $(iii)\Rightarrow(ii)$ in Theorem \ref{af} holds. Let $A$ and $\mathfrak A$ be separable and unital, and let $\mathfrak A$ be a continuous
$C(X)$-algebra with simple fibres, for some  totally disconnected, compact Hausdorff space $X$, then the above implication holds:  If ${\rm dim}_{\rm nuc}^\mathfrak A(A)=0$, then we also have ${\rm nd}_\mathfrak A(A)= 0 $, thus $A$ is $\mathfrak A$-NF, by \cite[Theorem 4.3.12]{g}.

$(ii)$ We give a class of examples in which the identity map ${\rm id}_A: A\to A$ approximately decomposes via $F\otimes \mathfrak A$, with $F$ finite dimensional and the upward maps approximately order 0. Let $A$ and $\mathfrak A$ be separable and unital, and let $\mathfrak A$ be a continuous
$C_0(X)$-algebra with simple fibres, for some  totally disconnected, locally compact, Hausdorff space $X$, then if ${\rm dr}_\mathfrak A(A)<\infty$, then we may always choose an approximate decomposition with upward maps approximately order 0: it is shown in \cite[Proposition 4.3.7]{g} that starting with an approximate decomposition, one could properly modify the upward maps to become approximately order 0. This modification involves multiplying downward maps with scalar-valued continuous functions on $X$ (which preserves admissibility) and using ``supporting homomorphisms'' of the upward maps (which exist in the case of module maps, by Lemma \ref{homo}).
If $\mathfrak C$ is the Cantor set, $\mathfrak A:=\mathbb K(C(\mathfrak C)\otimes\ell^2)=C(\mathfrak C)\otimes \mathbb K(\ell^2)$ is a concrete example.   
\end{example}  

Recall that a C*-algebra $A$ is {\it projective} if given a $*$-homomorphism 
$\phi: A\to B/I$, for a C*-algebra factor with quotient map $q: B\to B/I$, there is a $*$-homomorphism lift $\tilde\phi: A\to B$ with $q\circ\tilde\phi=\phi$ \cite{l}. Finite dimensional C*-algebras and their cones are known to be projective. More generally, the cone of a separable C*-algebra is a direct limit of projective C*-algebras \cite{ls}. We say that an $\mathfrak A$-C*-algebra $A$ is $\mathfrak A$-{\it projective} if for each $\mathfrak A$-C*-algebra $B$ and each closed ideal and submodule $J\unlhd B$, each $*$-homomorphism and module map $\phi: A\to B/J$ lifts to a $*$-homomorphism and module map $\tilde \phi: A\to B$ with $q\circ\tilde\phi=\phi$, where $q: B\to B/J$ is the quotient map. We say that  $A$ is $\mathfrak A$-{\it cone projective} if the cone $CA$ is  $\mathfrak A$-projective. Since $\mathbb M_n(C\mathfrak A)=C\mathbb M_n(\mathfrak A)$, it follows that if $\mathfrak A$ is $\mathfrak A$-cone projective, so is $\mathbb M_n(\mathfrak A)$. When $\mathfrak A=\mathbb C$, finite dimensional C*-algebras are $\mathbb C$-cone projective.      

\begin{example} \label{eseq}
	In this example, we are going to show  further properties which apparently hold only in the case where $\mathfrak A$ is $\mathfrak A$-cone projective. 
	
	$(i)$ If $\mathfrak A$ is $\mathfrak A$-cone projective, then  ${\rm dr}(A)={\rm dr}_{\mathfrak A}(\iota_A)={\rm dr}_{\mathfrak A}(A)$:
	Each approximate factoring of ${\rm id}_A$ gives one for $\iota_A$, hence, in general we have, ${\rm dr}_{\mathfrak A}(A)\geq{\rm dr}_{\mathfrak A}(\iota_A)$. For the reverse inequality, and for given finite set $\mathcal F\ssubset A$ and $\varepsilon > 0$,  finite dimensional C*-algebra $F = \mathbb M_{k(0)}\oplus \cdots\oplus \mathbb M_{k(n)}$, c.c.p. admissible map $\psi: A \to  F \otimes \mathfrak A$,  c.p. module map $\varphi:=\sum_{i=0}^{n} \varphi^{(i)}: F\otimes \mathfrak A\to A_\infty$, we have  $\|\varphi\circ\psi(a) - \iota_A(a)\|< \varepsilon$, for $a\in \mathcal F$, with $\varphi^{(i)}$ order 0 module map on $F^{(i)}\otimes \mathfrak A$.
	By Lemma \ref{homo}, there is a $*$-homomorphism module map $\pi^{(i)}: C\mathbb M_{k(i)}(\mathfrak A)\to A_\infty$ with $\pi^{(i)}({\rm id}\otimes x)=\varphi^{(i)}(x)$, for $x\in \mathbb M_{k(i)}(\mathfrak A)$. Since $\mathbb M_{k(i)}(\mathfrak A)$ is $\mathfrak A$-cone projective, the map $\pi^{(i)}$ has a lift to a $*$-homomorphism module map $\tilde\pi{(i)}:\mathbb M_{k(i)}(\mathfrak A)\to \ell^\infty(A)$ with $q\circ\tilde\pi{(i)}=\pi^{(i)}$, for the quotient map $q: \ell^\infty(A)\to A_\infty$. In particular, 
	$q\circ\pi{(i)}({\rm id}\otimes x)=\varphi^{(i)}(x),$ for each $x\in \mathbb M_{k(i)}(\mathfrak A).$    
	Let us write, $\pi{(i)}=(\pi_{k}^{(i)})_k$, with $\pi_{k}^{(i)}: \mathbb M_{k(i)}(\mathfrak A)\to A$  $*$-homomorphism and module map. Put $\tilde\varphi_{k}^{(i)}(x):=\pi_{k}^{(i)}({\rm id}\otimes x)$, for $x\in C\mathbb M_{k(i)}(\mathfrak A)$. Given $a\in \mathcal F$, let $\bar a:=(a,a,\cdots)\in\ell^\infty(A)$ be the corresponding constant sequence, then $\iota_A(a)=q(\bar a)$. Let $\psi(a)=(b^{(i)})_{i=0}^{n}=:(\psi^{(i)}(a))_{i=0}^{n}$, with $b^{(i)}\in  C\mathbb M_{k(i)}(\mathfrak A)$, and observe that,   
	\begin{align*}
		\varepsilon&\geq\|\iota_A(a)-\sum_{i=0}^{n}\varphi^{(i)}(\psi(a))\| =\|q(\bar a)-\sum_{i=0}^{n}q(\pi^{(i)}({\rm id}\otimes b^{i}))\|\\&\geq\| a-\sum_{i=0}^{n}\pi_{k}^{(i)}({\rm id}\otimes b^{i})\|-\varepsilon=\| a-\sum_{i=0}^{n}\varphi_k^{(i)}\circ\psi^{(i)}(a)\|-\varepsilon,
	\end{align*} 
	for large $k$. This plus the fact that each $\varphi_{k}^{(i)}$ is order 0, shows that ${\rm dr}(A)\leq {\rm dr}_\mathfrak A(\iota_A)$.

	$(ii)$ By a similar proof, one could conclude that when $\mathfrak A$ is $\mathfrak A$-cone projective, for each $*$-homomorphism and module map $\theta: A\to B$, we have 
	$${\rm dr}(\theta)={\rm dr}_{\mathfrak A}(\theta\circ\iota_B)= {\rm dr}_{\mathfrak A}(\theta).$$
	
	$(iii)$ In part $(i)$, if we could arrange the upward maps $\varphi^{(i)}: F\otimes\mathfrak A\to A_\infty$ to be c.p. order 0 {\it module maps}, then we also have,  ${\rm dim}_{\rm nuc}(A)\geq{\rm dim}_{\rm nuc}^\mathfrak A(\iota_A)$. The other inequality always holds for nuclear dimension as well. Similarly, if we could arrange the upward maps $\varphi^{(i)}: F\otimes\mathfrak A\to B_\infty$ to be c.p. order 0 {\it module maps}, then we also have, ${\rm dim}_{\rm nuc}(\theta)={\rm dim}_{\rm nuc}^{\mathfrak A}(\theta\circ\iota_B)= {\rm dim}_{\rm nuc}^{\mathfrak A}(\theta).$ Similar facts hold for any map  $\theta: A\to B$, as above.
	
	$(iv)$ Given an exact sequence 
	$$0\to J\to A\to A/J\to 0$$
	of $\mathfrak A$-C*-algebras with connecting maps being $*$-homomorphisms and module maps, such that both $r := {\rm dim}_{\rm nuc}^{\mathfrak A}(J)$ and $s := {\rm dim}_{\rm nuc}^{\mathfrak A}(A/J)$ are
	finite. Given $\mathcal F\subseteq A_{1,+}$
	and $\varepsilon > 0$,  choose a  c.p. $\varepsilon/5$-approximate decomposition of ${\it id}_{A/J}$ via 
$F\otimes\mathfrak A=(F^{(0)}\oplus\cdots\oplus F^{(s)})\otimes\mathfrak A$ for $q(\mathcal F)$, where $q: A\to A/J$ is the quotient map with downward and upward maps $\psi$ and $\varphi$, respectively, with order 0 restrictions $\varphi^{(j)}$ on $F^{(j)}\otimes \mathfrak A$. If $\mathfrak A$ is separable and $\mathfrak A$-cone projective, each $\varphi^{(j)}$ lifts to a c.p.c. order zero map
$\tilde\varphi^{(j)}: F^{(j)}\otimes\mathfrak A\to A$, and $\tilde\varphi:=\tilde\varphi^{(0)}+\cdots+\tilde\varphi^{(s)}$ is c.p. lift of $\varphi$, which is also a module map. 

Since the cone of each separable C*-algebra could be presented as a universal C*-algebra with generators and relations, using \cite[Theorem 14.1.4]{lo}, by an argument similar to that of \cite[1.2.3]{wi2}, there is a constant $\delta>0$ such that for each $d\in A_{1,+}$  satisfying $[d, \tilde\varphi^{(j)}(b)]<\delta\|b\|$, for $b\in q(\mathcal F)$, one can find a c.p.c. order 0 module map $\hat\varphi: F^{(j)}\otimes\mathfrak A\to A$ with 
$$\|\hat\varphi^{(j)}(x)-d^{\frac{1}{2}}\tilde\varphi^{(j)}(x)d^{\frac{1}{2}}\|<\frac{\varepsilon}{5(s+1)}\|x\|,\ \ (x\in F^{(j)}\otimes \mathfrak A).$$	Using a quasicentral approximate unit in $J$ for $A$, we get an element $h\in J_{1,+}$ with
$[(1 - h), \tilde\varphi^{j}(x)]<\delta\|x\|$, for $x\in F^{(j)}\otimes\mathfrak A,$ and 
$$\|h^{\frac{1}{2}}ah^{\frac{1}{2}}+(1-h)^{\frac{1}{2}}a(1-h)^{\frac{1}{2}}-a\|<\frac{\varepsilon}{5}, \ \|(1-h)^{\frac{1}{2}}\big(\tilde\varphi\circ\psi(q(a))-a\big)(1-h)^{\frac{1}{2}}-a\|<\frac{2\varepsilon}{5},$$
 for $a\in\mathcal F.$ Using the above observation for $d=1-h$ and $\hat\varphi:=\hat\varphi^{(0)}+\cdots\hat\varphi^{(s)}$, we have,
$$\|\hat\varphi(x)-(1-h)^{\frac{1}{2}}\tilde\varphi(x)(1-h)^{\frac{1}{2}}\|<\frac{\varepsilon}{5}\|x\|,\ \ (x\in F^{(j)}\otimes \mathfrak A).$$
Next, choose a c.p. $\varepsilon/5$-approximate decomposition of ${\it id}_{J}$ via 
$F^{'}\otimes\mathfrak A=(F^{('0)}\oplus\cdots\oplus F^{('r)})\otimes\mathfrak A$ for $h^{\frac{1}{2}}\mathcal Fh^{\frac{1}{2}}$,
with downward and upward maps $\psi^{'}$ and $\varphi^{'}$, respectively. Then we get a c.p. $\varepsilon$-approximate decomposition of ${\it id}_{A}$ via 
$(F\oplus F^{'})\otimes\mathfrak A$ for $\mathcal F$,
with downward and upward maps $\psi\circ q\oplus\psi^{'}\circ{\rm ad}_{h^{\frac{1}{2}}}$ and $\hat\varphi+\varphi^{'}$, respectively, that is,  ${\rm dim}_{\rm nuc}^{\mathfrak A}(A)\leq r+s+1$, as in the classical case (see, Remark \ref{es}$(iii)$). 

$(v)$ Let $X$ be a compact metrizable space, $H$ be a separable infinite
dimensional Hilbert space. By an $\mathfrak A$-{\it Toeplitz
algebra} we mean an extension $E$,
$$0\to \mathbb K(H\otimes\mathfrak A)\to E\to C(X)\otimes \mathfrak A\to  0.$$
If $\mathfrak A$ is $\mathfrak A$-cone projective, by Remark \ref{es}$(i)$ and part $(iv)$ above, for any $\mathfrak A$-Toeplitz algebra $E$, 
$$ {\rm dim}_{\rm nuc}^{\mathfrak A}(C(X)\otimes\mathfrak A)\leq {\rm dim}_{\rm nuc}^{\mathfrak A}(E)\leq  {\rm dim}_{\rm nuc}^{\mathfrak A}(C(X)\otimes\mathfrak A)+{\rm dim}_{\rm nuc}^{\mathfrak A}(\mathbb K(H\otimes\mathfrak A))+1.$$
Applying the estimate ${\rm dim}_{\rm nuc}^\mathfrak A(\mathfrak A\otimes B)\leq {\rm dim}_{\rm nuc}(B)$, to $B=C_0(X)$ and $\mathbb K(H)$, we get the upper bound, ${\rm dim}_{\rm nuc}^{\mathfrak A}(E)\leq  {\rm dim}(X)+1.$  When $\mathfrak A$ is separable and $X=\mathfrak C$ is the Cantor set, $C(\mathfrak C)\otimes\mathfrak A$ is $\mathfrak A$-AF, and by Theorem \ref{af}, $ {\rm dim}_{\rm nuc}^{\mathfrak A}(C(\mathfrak C)\otimes\mathfrak A)=0$, and we get  ${\rm dim}_{\rm nuc}^{\mathfrak A}(E)=0$ or $1$.
\end{example}

Let us next recall the notion of Rokhlin dimension for finite group actions on
C*-algebras. Let $G$ be a finite group, $\mathcal A$ a unital C*-algebra and
$\gamma : G \to {\rm Aut}(\mathfrak A)$ be
an action of $G$ on $\mathfrak A$. We say that $\gamma$ has Rokhlin dimension $d$, writing ${\rm dim}_{\rm Rok}(\gamma):=d$, if $d$ is the least integer such
that the following holds: for any $\varepsilon > 0$ and every finite subset $\mathcal L \subseteq \mathfrak A$, there are positive
contractions $\alpha_s^i\in \mathfrak A$, for $s\in G$, $i=0,\cdots, d$, satisfying:

(1) 
$\|\alpha_s^i\alpha_t^i\|<\varepsilon$, whenever $s\neq t$,

(2) 
$\|1_\mathfrak A-\sum_{i=0}^{d}\sum_{s\in G}\alpha_s^i\|<\varepsilon,$

(3) $\|\gamma_t(\alpha_s^i)-\alpha_{ts}^i\|<\varepsilon,$

(4) $\|[\beta, \alpha_s^i]\|< \varepsilon$,

\noindent for $s,t\in G$, $i=0,\cdots, d$, and $\beta\in\mathcal L$. 

By approximating the function $t\mapsto t^{\frac{1}{2}}$ by polynomials on $[0, 1]$ we may
always assume that we further have,

(5) $\|(\alpha_s^i)^{\frac{1}{2}}(\alpha_t^i)^{\frac{1}{2}}\|<\varepsilon$, whenever $s\neq t$, 

(6)
$\|\gamma_t\big((\alpha_s^i)^{\frac{1}{2}}\big)-(\alpha_{ts}^i)^{\frac{1}{2}}\|<\varepsilon,$
 
(7)
$\|[\beta, (\alpha_s^i)^{\frac{1}{2}}]\|< \varepsilon$,

\noindent for $s,t\in G$, $i=0,\cdots, d$, and $\beta\in\mathcal L$.
 
When $d = 0$, we say that $\gamma$ has Rokhlin property. If we further could guarantee that, 

(8) $(\alpha_s^i)^{\frac{1}{2}}\beta=\gamma_s(\beta)(\alpha_s^i)^{\frac{1}{2}},$

\noindent for  $s\in G$, $i=0,\cdots, d$, and $\beta\in\mathfrak A$, then we say that $\gamma$ has compatible Rokhlin dimension $d$, and write ${\rm dim}_{\rm Rok}^{\rm com}(\gamma):=d$.

We say that $\mathfrak A$ satisfies {\it admissible stability property for order zero maps} (ASOZ) if for each $n\geq 1$, given finite subset $\mathcal F\subseteq \mathbb M_n(\mathfrak A)$ and $\varepsilon>0$, there is $\eta>0$ such that for each $\mathfrak A$-C*-algebra $B$ and c.p. admissible map $\phi: \mathbb M_n(\mathfrak A)\to B$ with $\|\phi(a)\phi(b)\|<\eta$, for orthogonal elements $a,b\in\mathcal F$, there is a c.p.  order 0 admissible map $\hat\phi: \mathbb M_n(\mathfrak A)\to B$ with $\|\hat\phi(a)-\phi(a)\|<\varepsilon$, for $a\in\mathcal F$. If {\it admissible} map in the above definition is replaced everywhere by {\it module} map, we say that $\mathfrak A$ satisfies {\it module stability property for order zero maps} (MSOZ).
      
\begin{example}
Let $G$ be a finite group, $\mathfrak A$ be a unital C*-algebra satisfying (ASOZ) (resp., (MSOZ)) and $\gamma : G \to {\rm Aut}(\mathfrak A)$ be an action with (resp., compatible) finite Rokhlin dimension.
Then the crossed product $A:=\mathfrak A\rtimes_\gamma  G$ has finite $\mathfrak A$-nuclear dimension (resp., finite $\mathfrak A$-decomposition rank), and $${\rm dim}_{\rm nuc}^\mathfrak A(\mathfrak A\rtimes_\gamma G)\leq{\rm dim}_{\rm Rok}(\gamma) \ \ ({\rm resp.},\ \ {\rm dr}^\mathfrak A(\mathfrak A\rtimes_\gamma G)\leq{\rm dim}_{\rm Rok}^{\rm com}(\gamma)).$$   

To show this, let us set $n := |G|$ and $d:={\rm dim}_{\rm Rok}(\gamma) \ ({\rm resp.}, d:={\rm dim}_{\rm Rok}^{\rm com}(\gamma)$), and let the finite subset $\mathcal F\subseteq A:=\mathfrak A\rtimes_\gamma  G\subseteq \mathbb M_n(\mathfrak A)$ and $\varepsilon > 0$ be given. We may assume that $\mathcal F$
consists of contractions. Every element $a\in\mathcal F\subseteq \mathbb M_n(\mathfrak A)$ could be written as $a=\sum_{s\in G} a_su_s$, for uniquely determined contractions $a_s\in\mathfrak A$. Put $\mathcal L:=\{a_s: a\in\mathcal F, s\in G\}$. Given $\delta> 0$ and finite set $\mathcal L$, choose a multiple
tower $(\alpha_s^i)_{i=0}^{d}$ satisfying (1)–(7) (resp., (1)-(8)) above, for $\delta$ and $\mathcal L$. Consider the c.p. admissible (resp., module) map, $$\phi^{(i)}: \mathbb M_n(\mathfrak A)\to A;\ 
\phi^{(i)}(e_{s,t}\otimes\alpha) := (\alpha_s^i)^{\frac{1}{2}}u_s\alpha u_t^*(\alpha_t^i)^{\frac{1}{2}},$$
and put $\phi_0:=\sum_{i=0}^{d}\phi^{(i)}$, which is again a c.p. admissible map (resp., which is also a module map, since by (8),
\begin{align*}
	\phi^{(i)}(e_{s,t}\otimes\alpha\beta)&=(\alpha_s^i)^{\frac{1}{2}}u_s\alpha\beta u_t^*(\alpha_t^i)^{\frac{1}{2}}\\&=(\alpha_s^i)^{\frac{1}{2}}u_s\alpha u_t^*\gamma_{t}(\beta)(\alpha_t^i)^{\frac{1}{2}}\\&=(\alpha_s^i)^{\frac{1}{2}}u_s\alpha u_t^*(\alpha_t^i)^{\frac{1}{2}}\beta\\&=\phi^{(i)}(e_{s,t}\otimes\alpha)\beta,
\end{align*}	
as claimed). Next, observe that,
\begin{align*}
\phi^{(i)}(\alpha u_s)&= \sum_{t\in G}\phi^{(i)}\big(e_{t, s^{-1}t}\otimes\gamma_{t^{-1}}(\alpha)\big)\\&=\sum_{t\in G}(\alpha_t^i)^{\frac{1}{2}}u_t\gamma_{t^{-1}}(\alpha) u_{t^{-1}s}(\alpha_{s^{-1}t}^i)^{\frac{1}{2}}\\&=\sum_{t\in G}(\alpha_t^i)^{\frac{1}{2}}\alpha u_{s}(\alpha_{s^{-1}t}^i)^{\frac{1}{2}}\\&\approx_{n\delta}\sum_{t\in G}(\alpha_t^i)^{\frac{1}{2}}\alpha (\alpha_{s^{-1}t}^i)^{\frac{1}{2}}u_{s}\\&\approx_{n\delta}\sum_{t\in G}\alpha_t^i \alpha u_{s},
\end{align*}
thus, for $\delta_0:=2\delta((d+1)n+1)$,
$$\phi^{(i)}(\alpha u_s)\approx_{\delta_0}\alpha u_s,$$
for $s\in G$, $i=0,\cdots,d$, and $\alpha\in\mathcal L$, as $\mathcal L$ consists of contractions. Therefore, $\phi_0(a)\approx_{n\delta_0}a,$ for $a\in \mathcal F$. Let $1_n$ be  the identity of $\mathbb M_n(\mathfrak A)$, then since $$\|\phi_0(1_n)\|=\big\|\sum_{i=0}^{d}\sum_{s\in G} \alpha_s^{i}\big\|< 1 + \delta,$$ after normalization, we may assume that $\phi_0$ is c.c.p. with $\phi_0(a)\approx_{2n\delta_0}a,$ for $a\in \mathcal F$, that is, $\phi_0(a)\approx_{{\frac{1}{3}}\varepsilon}a,$ for $a\in \mathcal F$, when $\delta<\big(12n((d+1)n+1)\big)^{-1}\varepsilon$.  
Put $\alpha^i:=\sum_{s\in G}\alpha_s^{i}$, and observe that, given $\alpha,\beta\in\mathfrak A$ with and $s_1, s_2, t_1, t_2\in G$, for $\tilde\alpha:=e_{s_1,t_1}\otimes\alpha$ and $\tilde\beta:=e_{s_2,t_2}\otimes\beta$,  
$$\|\phi^{(i)}(\tilde\alpha)\phi^{(i)}(\tilde\beta)-\alpha^i\phi^{(i)}(\tilde\alpha\tilde\beta)\|< (n + 4)\delta.$$
Therefore, 
$$\|\phi^{(i)}(a)\phi^{(i)}(b)-\alpha^i\phi^{(i)}(ab)\|\leq (n + 4)n^3\delta,$$
for $a,b\in \mathcal F$. In particular, if $ab=0$, then $\|\phi^{(i)}(a)\phi^{(i)}(b)\|\leq (n + 4)n^3\delta.$ If $\eta>0$ be the constant as in (ASOZ) (resp., (MSOZ)) and $\delta<(n+4)^{-4}\eta$, we may replace $\phi^{(i)}$ with c.p. order 0 admissible (rep., module) map $\tilde\phi^{(i)}: \mathbb M_n(\mathfrak A)\to A$ such that for $\phi:=\sum_{i=0}^{d}\tilde\phi^{(i)}$, we have $\phi(a)\approx_{\frac{1}{3}\varepsilon}\phi_0(a)$, for $a\in\mathcal F$. Again, normalizing if needed, we may assume that $\phi$ is c.c.p. and $\phi(a)\approx_{\frac{2}{3}\varepsilon} a$, for $a\in\mathcal F$. Finally, if  $\psi: A\hookrightarrow \mathbb M_n(\mathfrak A)$ is the inclusion map, then $\phi\circ\psi(a)\approx_{\varepsilon} a$, for $a\in\mathcal F$, that is, ${\rm dim}_{\rm nuc}^\mathfrak A(\mathfrak A\rtimes_\gamma G)\leq d.$ ({\rm resp.},\  ${\rm dr}^\mathfrak A(\mathfrak A\rtimes_\gamma G)\leq d$). 
\end{example}

\section*{Author's statements}
{\bf Competing Interests}: There is no competing interest related to this article. 

{\bf Data availability}: Data sharing not applicable to this article as no datasets were generated or analysed during the current study.

\bibliographystyle{line}
\bibliography{JAMS-paper}

\end{document}